\documentclass{amsart}
\usepackage{graphicx} 
\usepackage{hyperref}
\usepackage{amsfonts,amsmath,amssymb}
\usepackage{mathrsfs}
\usepackage{amsthm}
\usepackage{thmtools}
\usepackage{comment}
\usepackage{mathtools, nccmath}
\usepackage{tikz}
\usepackage{cleveref}
\usepackage{soul}

\usepackage{xcolor}

\newcommand{\bbR}{\mathbb{R}}
\newcommand{\bbP}{\mathbb{P}}

\newcommand{\bbZ}{\mathbb{Z}}
\newcommand{\bbN}{\mathbb{N}}

\newcommand{\PP}{\mathbb{P}}

\newcommand{\G}{\mathcal{G}}
\newcommand{\W}{\mathcal{W}}

\newcommand{\ind}[1]{\mathbf{1}_{\{#1\}}}

\newtheorem{proposition}{Proposition}[section]
\newtheorem{theorem}[proposition]{Theorem}
\newtheorem{lemma}[proposition]{Lemma}

\newtheorem{question}[proposition]{Question}

\theoremstyle{definition}
\newtheorem{definition}[proposition]{Definition}
\newtheorem{remark}[proposition]{Remark}

\usepackage[backend=biber, sorting=nyt, citestyle=alphabetic, bibstyle=alphabetic, isbn=false, doi=true, url=false, maxnames=7]{biblatex}
\title{A new bound for the critical point of the FK model for $q<1$}
\author{Vincent Beffara \and Corentin Faipeur \and Tejas Oke}
\address{Institut Fourier}
\email{vincent.beffara@univ-grenoble-alpes.fr}
\address{ENS Lyon}
\email{corentin.faipeur@ens-lyon.fr}
\address{California Institute of Technology, Pasadena}
\email{tejas@caltech.edu}
\date{\today}

\begin{document}
	
	\begin{abstract}
		We consider the random cluster model with parameter $q<1$, for which the FKG inequalities are not valid. On the square lattice, stochastic comparison with Bernoulli percolation implies that the model is subcritical (respectively supercritical) when $p \leq q/(1+q)$ (resp. $p \geq 1/2$); in this paper, we extend these two regions, by improving the classical stochastic comparisons.
		Assuming the existence of the critical point, this reduces its possible range.
		
		The proof relies on a modification of the usual Glauber dynamics of the model, which enables stochastic bounds of FK measures between two inhomegenous percolations.
		We also prove uniqueness of the infinite-volume measure in our extended ranges.
		Most of our results are valid in any dimension $d \geq 2$ and beyond hypercubic lattices.
	\end{abstract}
	
	\maketitle

	\begin{figure}[h!]
		\centering
		\includegraphics[width=1\linewidth]{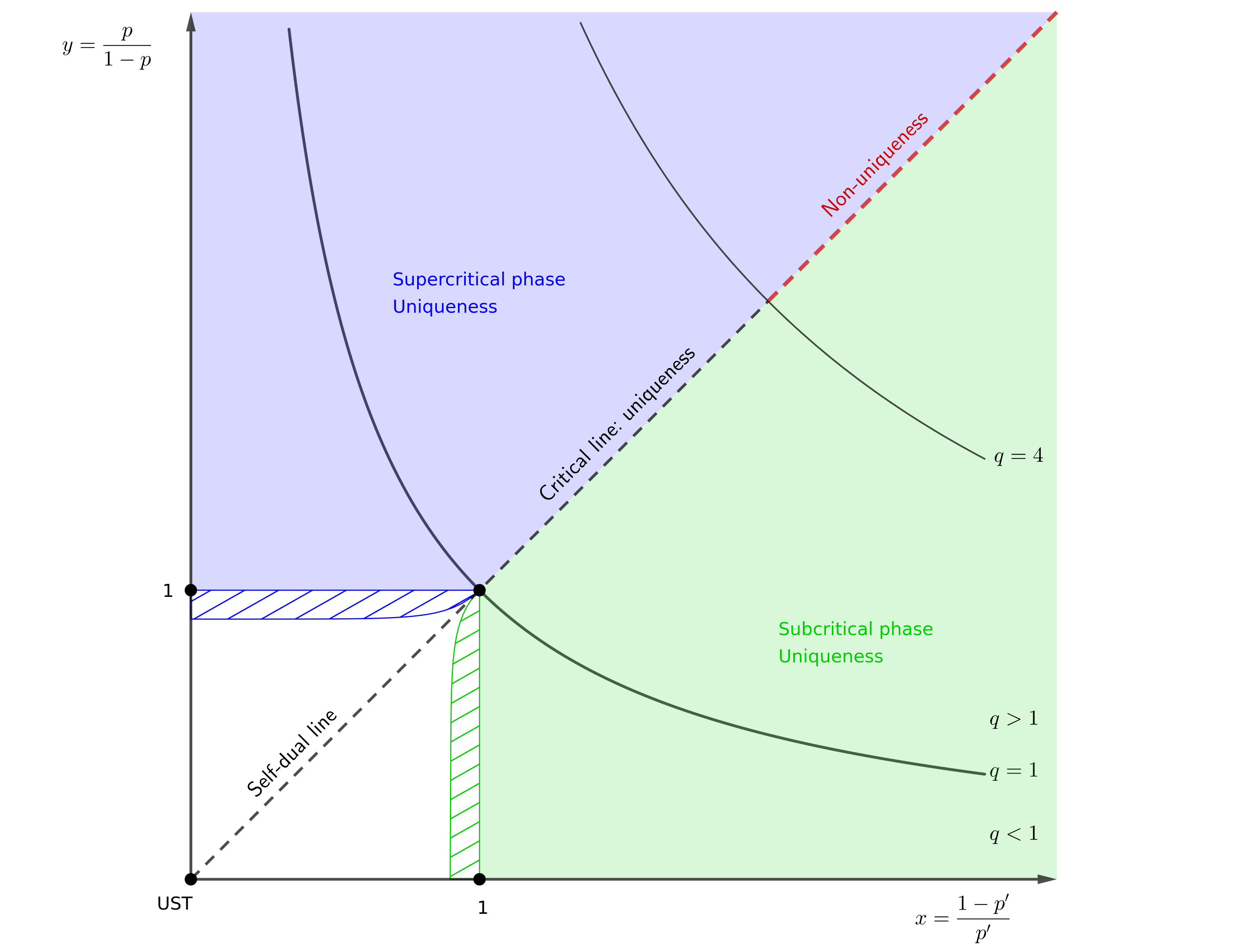}
		\caption{Phase diagram of the FK-percolation model on $\mathbf{Z}^2$.}
		\label{fig:phase_diag_Z2}
	\end{figure}
	
	\tableofcontents
	\addtocontents{toc}{\protect\setcounter{tocdepth}{1}}
	
	\section{Introduction and main results}
	\subsection{Introduction}
	
	The Random Cluster Model or FK-percolation, ‘FK’ for Fortuin-Kasteleyn, is a statistical mechanics model that was introduced in the early 1970s as a unifying graphical representation of percolation, the Ising model, and the $q$-state Potts model (see \cite{FortuinKasteleyn}, \cite{grimmett2006random}). On a finite graph $G$ it is defined to produce a random subgraph, where we put weights on the included edges and the number of connected clusters formed. It is characterized by two parameters $p$ and $q$, where $p$ is the `edge weight' and $q$ is the `cluster weight'.

	More specifically, the random cluster model on a finite graph is a probability measure on `edge configurations' (characterizing the subset of edges we include in our random subgraph, also called open edges, and the remaining edges called closed). The measure of a configuration is proportional to
	$$p^{\#\text{open edges}}\times(1-p)^{\#\text{closed edges}}\times q^{\#\text{open clusters}},$$
	where an open cluster (or simply a cluster) is a maximal connected subgraph formed by open edges. It is `similar' to percolation except that there is a weight on the number of clusters, making it a dependent percolation model. The case $q=1$ is exactly the classical Bernoulli percolation.
	
	One of the central concepts studied in statistical mechanics is that of \textit{phase transition}. During a phase transition of a given system, certain properties of the system change as a result of the change in external conditions. We are interested in studying the phase diagram of the Random Cluster Model on $\bbZ^d$ for $d\geq1$ (although, as we shall see, many of the results we prove carry over to a wider class of graphs).
	When we study the random cluster model on a box (a finite subgraph of $\bbZ^d$), these external conditions correspond to the connectivity properties of the boundary (which might affect the number of open clusters which are weighted by the FK measure).

	The model cannot be defined directly on $\bbZ^d$, but we can consider weak limits of finite-dimensional measures, with some boundary condition and the boxes growing to the whole lattice. Thus, the question of phase transition boils down to identifying the set of parameters for which we have uniqueness of the limiting measure in infinite volume. Having such a uniqueness would be the same as saying that changing external conditions does not affect the (limiting) measure.
	
	From the percolation point of view, phase transition relates to the presence or not of an infinite cluster in the random subgraph formed by open edges.
	It is generally true for Bernoulli percolation ($q=1$) that for $p$ small enough, almost surely all the clusters are finite, whereas for $p$ large enough, there is almost surely an infinite cluster. More precisely, on any infinite graph $G$, we can define a critical parameter $p_c(G) \in [0,1]$ above which there is almost surely an infinite cluster for the Bernoulli percolation on $G$, but below which there is almost surely no infinite cluster. The percolation is said to be supercritical in the first case, and subcritical in the second.
	It is known that for all $d \geq 2$, $p_c(\bbZ^d) \in (0,1)$, so that there is a subcritical and a supercritical phase (in contrast, $p_c(\bbZ)=1$).
	
	For the random cluster model with $q\geq1$, the FKG inequalities enable to define an analogous critical parameter, and the exact value is known in dimension 2 (see \cref{Phase diagram of the FK model}).
	For $q<1$, the situation is less understood due to the lack of these FKG inequalities. The definition of the critical point is not rigorous and only crude bounds on its value are known, and uniqueness of the infinite-volume measure is only known for a restricted range of parameters.
	The main goal of this paper is to contribute to the study of the model in this $q<1$ regime, improving the existing bounds on the critical parameter and extending the known regime of uniqueness.
	
	\subsection{Definition of the model and comparison inequalities}
	In this section, we will give concrete definition of the model on finite graphs with boundary conditions, and see some of its properties. A complete treatment can be found in \cite{grimmett2006random}. 
	
	Let $G = (V,E)$ be a finite connected graph, and $\partial G$ a subset of vertices, called the \emph{boundary} of $G$. Intuitively, one should think of $G$ as a finite box of an infinite (connected, locally finite) graph, e.g. the $d-$dimensional lattice, $\partial G$ being the set of vertices of $G$ having at least one neighbor outside $G$.
	Let $p \in [0,1], q \in (0, \infty)$ be the two parameters of the model. A configuration $\omega$ of the random cluster model on $G$ is an element $\omega$ of $\Omega = \{0,1\}^E$. An edge $e \in E$ is said to be open if $\omega(e) = 1$ and closed otherwise. Configurations can naturally be identified with subgraphs of $G$ with the same vertex set, with edge set corresponding to open edges. The inclusion (for the subgraphs) yields a natural partial order on $\Omega$.
	
	Two vertices $x$ and $y$ are said to be connected in a certain configuration if there exists a path of open edges from $x$ to $y$.
	The connectivity properties of a configuration will depend on the \emph{boundary condition} that we impose on $G$.
	Let $\alpha \in \Pi(\partial G)$ be a partition of $\partial G$.
	Two vertices that are in the same block in $\alpha$ are said to be \emph{wired}, which intuitively means that they are connected externally without using edges from $G$. The graph $G\cup \alpha$ is obtained from $G$ by identifying wired vertices; this identification, or \emph{wiring}, can create self-loops in the graph, which are anyways not excluded from our definition.
	For any configuration $\omega$, denote by $k(\omega,\alpha)$ the number of connected components in $G \cup \alpha$ --- or equivalently, the connected components obtained when starting from those of $\omega$ and then recursively identifying any two components sharing vertices in the same block of $\alpha$; it means that the open clusters of two wired vertices only count as 1 in $k(\omega, \alpha)$.
	The \emph{random cluster measure} or \emph{FK measure} on $G$ with parameters $p$ and $q$ and boundary condition $\alpha$ is the probability measure on $\Omega$ defined by
	\begin{equation}\label{eq:FK_measure_finite_graph}
		\phi^\alpha_{G,p,q}(\omega) := \frac{1}{Z^\alpha_{G,p,q}}\left(\prod_{e \in E}p^{\omega(e)}(1-p)^{1-\omega(e)}\right)q^{k(\omega,\alpha)}, \ \omega \in \{0,1\}^E
	\end{equation}
	where $Z^\alpha_{G,p,q}$ is the normalizing constant, called the partition function.
	The case $q=1$ is the classical Bernoulli bond percolation.
	There are two extremal boundary conditions: the first, called \emph{free boundary condition}, is obtained by taking the partition $(\{x\})_{x \in \partial G}$, which means that nothing happens outside $G$; the second corresponds to the opposite situation where all vertices are wired together, i.e. we take $\{\partial G\}$ as partition, and is called the \emph{wired boundary condition}.

	\begin{remark}
		The FK measure $\phi^\alpha_{G,p,q}$ with arbitrary boundary condition $\alpha$ can always be seen as an FK measure with free boundary condition on the graph $G\cup \alpha$.
	\end{remark}
	
	An event $A$ is said to be \emph{increasing} if it is preserved by the addition of open edges, i.e. if for all $\omega\leq \omega'$, $\omega \in A$ implies $\omega' \in A$.
	For two measures $\mu, \mu'$ on $\{0,1\}^E$, say that $\mu'$ \emph{stochastically dominates} $\mu$ if for every increasing event $A$, $\mu(A) \leq \mu'(A)$; in this case, write $\mu \preceq \mu'$. Equivalently, $\mu \preceq \mu'$ if and only if there is a coupling $(\omega,\omega')$ of $\mu$ and $\mu'$ such that $\omega \leq \omega'$ almost surely.

	For every configuration $\omega \in \{0,1\}^E$ and any edge $e\in E$, denote by $\omega_{\langle e\rangle}$ the restriction of $\omega$ to $E\setminus\{e\}$;
	$K^\alpha_e$ denotes the event that the endpoints of $e$ are connected by an open path not using $e$, but possibly using an imaginary edge between two wired vertices of the boundary. Note that the event $K^\alpha_e$ is independent of the state of the edge $e$.
	From \eqref{eq:FK_measure_finite_graph}, one can deduce that for every $e \in E$,
	\begin{equation}\label{eq:conditional_probas}
		\phi^\alpha_{G,p, q}\left(\omega(e)=1 \mid \omega_{\langle e\rangle}\right)= \begin{cases}p & \text { if } \omega_{\langle e\rangle} \in K^\alpha_e \\ \frac{p}{p+q(1-p)} & \text { if } \omega_{\langle e\rangle} \notin K^\alpha_e\end{cases}.
	\end{equation}
	For simplicity of notation, we write $p':=\frac{p}{p+q(1-p)}$.
	It turns out that $\phi^\alpha_{G, p,q}$ is the unique probability measure on $\Omega$ having the above conditional probabilities.
	If $G$ is a tree and $\alpha$ is the free boundary condition, it is clear that $K^\alpha_e$ cannot be achieved, so the FK measure is simply the Bernoulli product measure with parameter $p'$.
	In the other way, if all the vertices are wired together (or if $V$ is already reduced to a singleton), $K^\alpha_e$ is always achieved and the FK measure is the Bernoulli product measure with parameter $p$.
	
	From \eqref{eq:conditional_probas}, using a standard monotone coupling, it can be seen that we have the following comparison inequalities of the random cluster measure with product measures:
	\begin{equation}\label{eq:comparison_inequalities}
		\PP^G_{\min(p,p')} \preceq \phi_{G,p,q}^\alpha \preceq \PP^G_{\max(p,p')}
	\end{equation}
	where for $s\in[0,1]$, $\PP^G_s$ denotes the Bernoulli bond percolation measure on $G$ of parameter $s$ (this is the measure on graph $G$ where we include every edge independently with probability $s$).
	Note that $p<p'$ if and only if $q<1$ (ignoring the trivial cases $p=p'=0$ and $p=p'=1$).

	More generally, for $(p_e)_{e \in E} \in [0,1]^E$, let $\PP^G_{(p_e)}$ be the product measure associated with the inhomogeneous Bernoulli percolation on $G$, where edge $e$ is open with probability $p_e$. We prove that the comparison inequalities \eqref{eq:comparison_inequalities} can be enhanced:
	\begin{theorem}[Enhanced comparison inequalities]\label{thm:enhanced_comparison_ineq}
		Let $G=(V,E)$ be a finite connected graph with boundary $\partial G$, and $\alpha \in \Pi(\partial G)$ be a boundary condition. Assume that $G \cup \alpha$ is not a tree and contains at least two vertices (after wiring). Then, for all $p\in(0,1)$ and $q\neq1$, there exist $(\varepsilon_e)_{e \in E} \in [0,1]^E$ and $(\varepsilon'_e)_{e \in E} \in [0,1]^E$, which are both non-identically 0, such that
		\[\PP^G_{\min(p,p')+\varepsilon_e} \preceq \phi_{G,p,q}^\alpha \preceq \PP^G_{\max(p,p')-\varepsilon'_e}.\]
	\end{theorem}
	
	\begin{remark}
		The condition $G\cup \alpha$ is not a tree is satisfied if $G$ itself is not a tree or if $\alpha$ is not the free boundary condition; the condition $G\cup \alpha$ contains two vertices amounts to the fact that $G\cup \alpha$ is not reduced to a single point with self-loops. For planar graphs, these two conditions are dual of each other.
	\end{remark}
	
	\begin{figure}[ht]
		\centering
		\begin{minipage}[t]{0.48\textwidth}
			\centering
			\includegraphics[width=\linewidth]{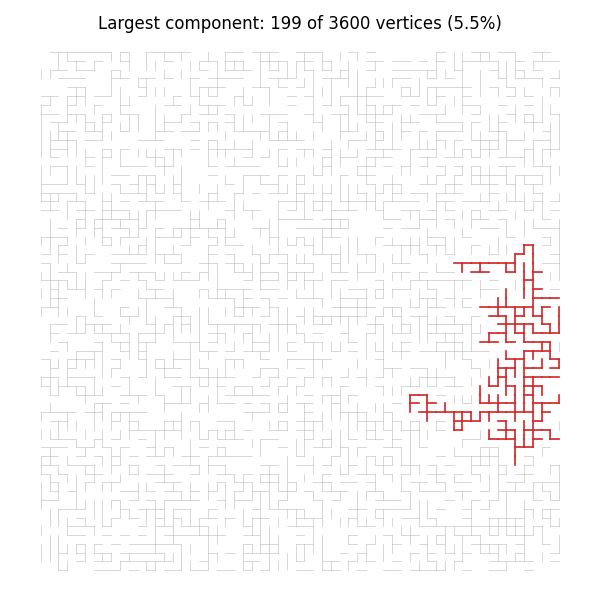} 
		\end{minipage}\hfill
		\begin{minipage}[t]{0.48\textwidth}
			\centering
			\includegraphics[width=\linewidth]{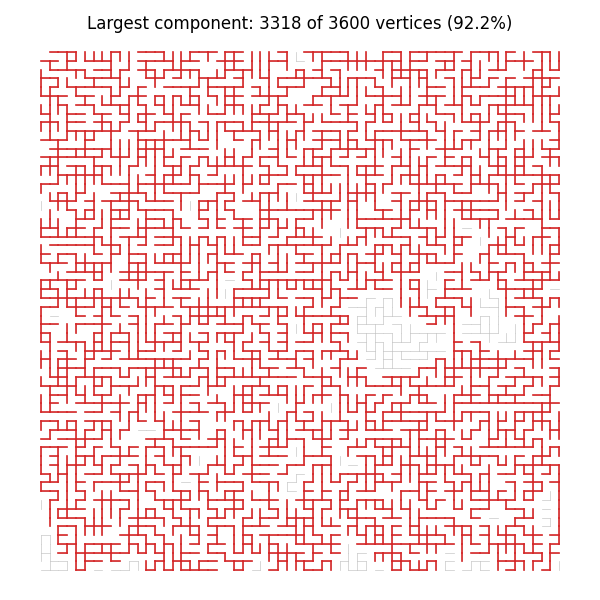}
		\end{minipage}
		\caption{Left: Simulation of percolation at $p = 0.45$.
			Right: Heat bath simulation\protect\footnotemark of FK model at same $p$, with $q = 0.15$.}
		
	\end{figure}
	\footnotetext{Reader should take this simulation with a pinch of salt because we dont have boundary FKG property to show fast mixing, so the configurations shown are obtained by running a large, fixed number of sweeps and should be read as empirical illustrations only.}

We aim to extend this result to infinite graphs, but one should notice before that the definition of the random cluster model on an infinite graph is not immediate.
If $G=(V,E)$ is an infinite, connected, locally finite graph, a way to obtain an FK measure on $G$ is by taking the so called \emph{thermodynamic limit}.
More precisely, say that a probability measure $\phi$ on $\Omega=\{0,1\}^E$
is an \emph{FK limit measure} on $G$ with parameters $p,q$ if there exist an increasing sequence of subgraphs $\Lambda_n \uparrow G$ and a sequence of partitions $\alpha_n \in \Pi(\partial \Lambda_n)$ such that
\begin{equation}\label{weak_limit}
	\phi=\lim_{n \to \infty} \phi^{\alpha_n}_{\Lambda_n,p,q}
\end{equation}
(for the weak convergence).
By compactness of $\Omega$, there always exists such an FK limit measure for any $p \in[0,1]$ and $q>0$ (say by the Prokhorov's theorem).
It is clear that \eqref{eq:comparison_inequalities} passes to the limit so that the comparison inequalities hold for any FK limit measure.
Therefore, if $\max(p,p')$ is less than $p_c(G)$ the critical parameter of the classical Bernoulli percolation on $G$, it follows that every FK limit measure is dominated by a subcritical percolation, so it is concentrated on configurations without infinite clusters. The dual statement holds for the supercritical case.

The main reason for proving an enhanced version of these comparison inequalities is to enlarge the set of parameters for which we have such a domination by a subcritical percolation.
To obtain a version of Theorem \ref{thm:enhanced_comparison_ineq} for infinite graphs, one needs to ensure that the $(\varepsilon_e)$ and $(\varepsilon_e')$ do not vanish when the size of the graph goes to infinity.
In fact, we prove more quantitative statements in Section \ref{Enhanced comparison inequalities for general finite graphs} (see Propositions \ref{prop:below_enhancement} and \ref{prop:above_enhancement}), with explicit values for $(\varepsilon_e)$ and $(\varepsilon_e')$.
Moreover, they will be positive for a positive density of edges. The theory of ``essential enhancements'' developed by Aizenman and Grimmett \cite{AizenmanGrimmett} then states that $\PP_{\max(p,p')-\varepsilon'_e}$ is still subcritical if $\max(p,p')$ is slightly above $p_c(G)$, likewise for the supercritical case. We therefore get the following result, which is illustrated by Figure \ref{fig:phase_diagram_general}:

\begin{theorem}[Extended subcritical and supercritical phases]\label{thm:new_bound_pc(q)}
	Let $G$ be an infinite, locally finite, connected, quasi-transitive graph such that $p_c(G)<1$, and assume that $G$ contains an edge which is on a cycle and on a doubly-infinite path.
	Then for all $q\neq 1$, there exists $\delta, \delta'>0$ such that:
	\begin{enumerate}
		\item If $\max(p,p')< p_c(G)+\delta'$, then there exists a subcritical percolation which dominates all the FK limit measures on $G$; in particular, for each, a.s. there is no infinite cluster.
		\item If $\min(p,p')> p_c(G) - \delta$, then there exists a supercritical percolation dominated by all the FK limit measures on $G$; in particular, for each, a.s. there is an infinite cluster.
	\end{enumerate}
\end{theorem}

\begin{remark}
	The assumption $p_c(G)<1$ is needed in order to apply the essential enhancements theory. Of course, it does not make sense to take $\max(p,p')$ slightly above 1, but notice that in the other direction, it is not true that taking $p$ slightly below $p_c(G)=1$ can make $\PP^G_{p+\varepsilon_e}$ supercritical (except if $\varepsilon_e>0$ for almost all $e$, which cannot be the case with our method). In fact, having $p_c(G)=1$ is a strong geometric feature of the graph that cannot be broken by local enhancements.
	We symmetrically need to have $p_c(G)>0$, but this is the case for all graphs of bounded degree (this can easily be seen by comparison with a Galton--Watson process).
\end{remark}

\begin{figure}[ht]
	\centering
	\includegraphics[width=1\linewidth]{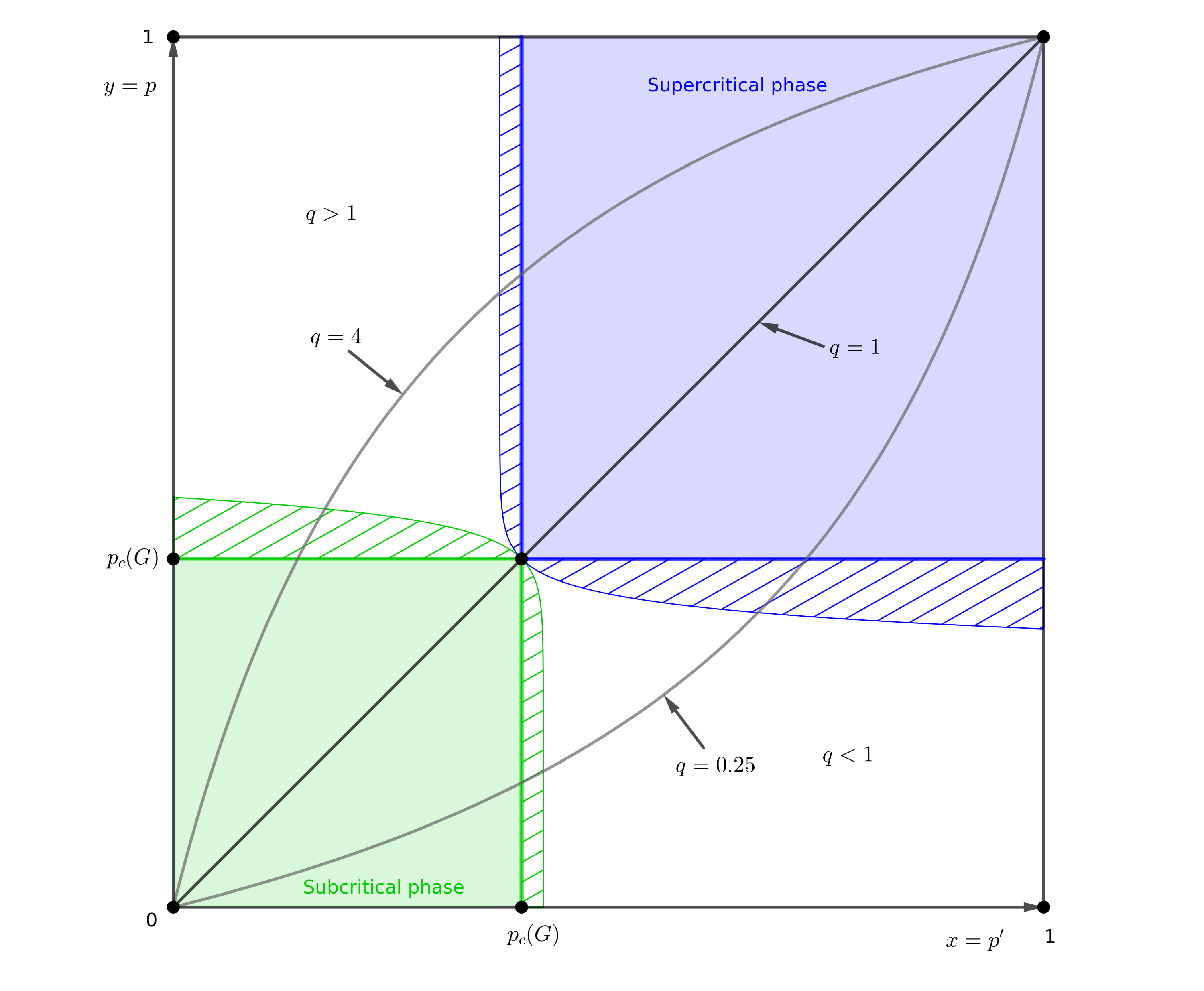}
	\caption{Phase diagram of the FK model on a graph which satisfies the hypothesis of Theorem \ref{thm:new_bound_pc(q)}.}
	\label{fig:phase_diagram_general}
\end{figure}

\begin{remark}
	The above result is stated for quasi-transitive graphs, which are those for which the automorphism group partitions the graph into finitely many orbits.
	We do not believe that this assumption is strictly necessary.
	The crucial geometric properties on the graph that we are using in the proof are on one hand that it has bounded degree, and on the other hand that it contains special edges which are on a doubly-infinite path and on a cycle of bounded length, and moreover that these special edges are sufficiently spread in the graph (more precisely, we need that each edge of the graph is at bounded distance from such a special edge).
	We thus believe that the statements remain true for some non-quasi-transitive graphs.
\end{remark}
There is actually another way of defining infinite-volume measures for the random cluster model, given by the DLR approach. The above theorem extends to this new class of FK measures, that we properly define in the next section.
Afterwards, we will be in a position to discuss uniqueness of the infinite-volume measure.

\subsection{Specifications for the Random Cluster model}

Before formally stating our main result, a few general considerations on Gibbs measures for the FK model are needed. Let $G=(V,E)$ be an infinite, connected, locally finite graph, and fix $\Lambda \subset G$ a finite subgraph; denote by $\partial \Lambda$ the set of vertices in $\Lambda$ having at least one neighbor outside it. 
A configuration of the random-cluster model inside $\Lambda$ is an element of $\{0,1\}^{E(\Lambda)}$, where $E(\Lambda)$ denotes the set of edges whose both endpoints are in $\Lambda$.

By a standard computation, the probability measures $\phi_{\Lambda,p,q}^{\alpha}$ satisfy the following compatibility relation: given $\Lambda \subset \Delta$, a partition $\alpha$ on $\partial\Delta$ and a configuration $\eta$ on $\Delta \setminus \Lambda$, the conditional distribution, under $\phi_{\Delta,p,q}^{\alpha}$, of $\omega_{|\Lambda}$ given $\omega_{|(\Delta \setminus \Lambda)} = \eta$ is equal to $\phi_{\Lambda,p,q}^{\alpha'}$ where $\alpha' =: \pi(\alpha,\eta)$ is the partition of $\partial \Lambda$ induced by $\eta$ and $\alpha$. In other words, the $\phi_{\Lambda,p,q}^{\alpha}$ behave like a specification, with one important difference: it is not the case in general that every partition $\alpha$ on $\partial \Lambda$ can be generated by a configuration outside $\Lambda$. 

To form a specification in the usual sense, one needs to fix a \emph{convention} to form partitions from configurations, that is, a family of measurable maps \[\Phi_\Lambda : \{0,1\}^{E(\Lambda)^c} \to \Pi(\partial\Lambda)\] so that the conditional distribution inside $\Lambda$ given a configuration $\omega$ outside $\Lambda$ can be defined as $\phi_{\Lambda, p, q}^{\Phi_\Lambda(\omega)}$. To preserve the compatibility relation above, such a convention needs to itself satisfy the following: for every $\Lambda \subset \Delta$ and any configurations $\omega$ on $\Delta^c$ and $\eta$ on $\Delta \setminus \Lambda$, \[\Phi_\Lambda(\omega \cup \eta) = \pi(\Phi_{\Delta}(\omega),\eta).\] This implies in particular that
\begin{itemize}
	\item Any two vertices of $\partial\Lambda$ that are connected by a path in $\omega$ outside $\Lambda$ are in the same block of $\Phi_\Lambda(\omega)$;
	\item Any two vertices of $\partial\Lambda$ belonging to two disjoint clusters of $\omega$ in $\Lambda^c$, at least one of which being finite, are in different blocks of $\Phi_\Lambda(\omega).$
\end{itemize}
In other words, $\Phi_\Lambda$ only describes which pairs of disjoint infinite clusters outside $\Lambda$ are ``joined at infinity,'' in a measurable way. We will say that a probability measure $\phi$ on configurations in $G$ is an \emph{FK $\Phi$-Gibbs measure} with parameters $p$ and $q$ (or shortly is $\Phi$-Gibbs) if it is a Gibbs measure for the specification $\big(\phi_{\Lambda,p,q}^{\Phi_\Lambda(\cdot)}\big)_\Lambda$, that is, $\phi$ satisfies the DLR equations:
\begin{equation}\label{def: DLR-specification}
	\forall \Lambda \Subset G, \forall \xi \in \{0,1\}^E, \  \phi(\cdot \mid \mathcal{F}_{E(\Lambda)^c})(\xi) = \phi^{\Phi_\Lambda(\xi)}_{\Lambda, p, q}(\cdot) \quad \text{for } \phi-\text{a.e. } \xi.
\end{equation}
It is a general fact that a positive specification is characterized by its singleton part, see \cite{Georgii} Theorem 1.33.
Hence, for $\phi$ to be $\Phi$-Gibbs, it is enough to have for $\phi-$almost every $\omega$,
\[\phi(\omega(e)=1 \mid \mathcal{F}_{\langle e \rangle})(\omega)=\begin{cases}p & \text { if } \omega \in K^\Phi_e \\ \frac{p}{p+q(1-p)} & \text { if } \omega \notin K^\Phi_e\end{cases}\]
where $\mathcal{F}_{\langle e \rangle}$ is the $\sigma-$algebra generated by the states of all edges except $e$, and $K_e^\Phi$ is the event that both endpoints of $e$ lie in the same block in the partition $\Phi_{\{e\}}(\omega_{\langle e\rangle} )$; it means that either they are connected in the configuration $\omega_{\langle e\rangle}$ or that they belong to two distinct infinite clusters that are joined at infinity by the convention $\Phi$.

Two important examples of such are the \emph{free convention} $\Phi^f$ for which two vertices of $\partial\Lambda$ are in the same block of $\Phi^f_\Lambda(\omega)$ if and only if they are connected by a path in $\omega$ (\emph{i.e.}, nothing happens at infinity) and the \emph{wired convention} $\Phi^w$ for which two vertices of $\partial\Lambda$ are in the same block of $\Phi^w(\omega)$ if and only if they are connected by a path in $\omega$ or they are both in infinite clusters of $\omega$ (\emph{i.e.}, all infinite clusters are merged). If $G$ is a tree, every measurable partition of the end of $G$ induces a different convention.

\bigskip

In general, there is no reason that the set of all $\Phi$-Gibbs measures is the same for different conventions. For instance, on a tree, the only $\Phi^f$-Gibbs measure is the product measure with parameter $p'$ while for $\Phi^w$, every measure having doubly-infinite clusters will be non-product (and such measures do exist for $p$ close enough to $1$). On the other hand, one sees directly from the definition above that if a measure $\phi$ is $\Phi$-Gibbs and is such that $\phi$-a.s. there is at most one infinite cluster, then it is $\Psi$-Gibbs for every other convention $\Psi$ as well: typically, on an amenable lattice, this is the case for all translation-invariant measures, by an argument of Burton and Keane \cite{BurtonKeane}.
It is also true (see \cite{grimmett2006random} Theorem 4.31) that if an FK limit measure satisfies that there is at most one infinite cluster a.s., then it is an FK Gibbs measure (for any convention).

Thus, in general, different conventions give rise to different sets of Gibbs measures.
One might believe that weak limits of configurations on $\Lambda$ with empty boundary condition (for which all vertices of $\partial\Lambda$ are in different blocks) are $\Phi^f$-Gibbs, and that weak limits of configuration with wired boundary condition (for which all vertices of $\partial\Lambda$ are in the same block) are $\Phi^w$-Gibbs.
Unfortunately, this is not true in general as underlined by Halberstam and Hutchcroft in \cite[Remark 9]{HalberstamHutchcroft}: if one considers the \emph{uniform spanning tree}, it is known since the seminal work of Pemantle \cite{Pemantle} that on $\bbZ^d$, for all $d\geq1$, the free and wired limit measures (are well-defined and) always coincide; moreover, for $d\geq 5$, the measure is supported on spanning forests containing infinitely many trees, so that it is not a Gibbs measure for the free convention (if it was the case, for any finite box intersecting several infinite trees, these trees would be connected inside the box with probability 1); more trivially, the measure on $\bbZ$ is not a Gibbs measure for the wired convention (since otherwise every finite box would contain a.s. one closed edge).
We believe that this is a direction worth exploring, where interesting behavior will be found for the FK model in the case $q<1$; but it is outside the scope of the present article.

\begin{remark}
	Halberstam and Hutchcroft in \cite{HalberstamHutchcroft} use augmented subgraphs to define Gibbs measures for the \emph{arboreal gas}, and their formalisms also extend to define Gibbs measures for the random cluster model. One benefit of their definitions, since they allow random boundary conditions, is that Gibbs measures coincide with weak limits. We don't need this general formalism in our proofs, since the regimes where our proofs work already have the `at most one infinite cluster' property.
\end{remark}

\bigskip

In the following, for a convention $\Phi$, we denote by $\G^\Phi_{p,q}$ the set of $\Phi-$Gibbs measures for the random cluster model with parameters $p$ and $q$ on the graph $G$. Let $\W_{p,q}$ be the set of FK (weak) limit measures, and 
$$\operatorname{FK}(p,q)= \bigcup_{\Phi \text{ a convention}} \G^\Phi_{p,q} \cup \W_{p,q}$$
the set of all FK measures on $G$ with parameter $p$ and $q$.
It is clear from the definition of a Gibbs measure that $\G^\Phi_{p,q}$ is always a convex set (or empty). We write $\operatorname{ex}\G^\Phi_{p,q}$ for the set of its extremal elements, which are called extremal Gibbs measure.
It is a general fact that extremal Gibbs measures can be obtained as a thermodynamic limit along any sequence of boxes, with some boundary conditions, see \cite{friedli_velenik_2017}.
We gather this property together with another aforementioned fact about FK measures in the following Proposition:
\begin{proposition}\label{prop:inclusion_FK_measures}
	For all $p\in[0,1]$ and $q>0$, we have:
	\begin{enumerate}
		\item For any convention $\Phi$, every $\phi \in \operatorname{ex} \G^\Phi_{p,q}$ satisfies $$\phi=\lim_{n \to \infty} \phi^{\Phi_{\Lambda_n}(\omega)}_{\Lambda_n,p,q} \text{ for $\phi-$a.e. $\omega$}$$ for all sequence of boxes $\Lambda_n \uparrow G$; in particular, $\operatorname{ex} \G^\Phi_{p,q} \subset \W_{p,q}.$
		\item If $\phi \in \operatorname{FK}(p,q)$ is such that $\phi-$a.s. there is at most one infinite cluster, then $\phi \in \G^\Phi_{p,q}$ for all $\Phi$.
	\end{enumerate}
\end{proposition}

For the purpose of what follows, we will say that \emph{the random-cluster at parameters $p$ and $q$ on the graph $G$ exhibits uniqueness} if for every choice $\Phi$ of convention, the set $\G^\Phi_{p,q}$ has a single element, and that it exhibits \emph{strong uniqueness} if $\operatorname{FK}(p,q)$ has a single element. According to item (1) of the above Proposition, strong uniqueness is equivalent to the uniqueness of the FK limit measure, as $|\operatorname{ex} \G^\Phi_{p,q}|=1$ implies $|\G^\Phi_{p,q}|=1$ (this is an application of the stronger fact that $\G^\Phi_{p,q}$ is a simplex, see \cite{Georgii} Theorem 7.26).

For instance, the usual Peierls argument shows that if $G$ has bounded degree and $q>0$ is fixed, strong uniqueness occurs for every $p$ small enough (and the unique Gibbs measure is concentrated on non-percolating configurations). On the other hand, on a regular tree of degree at least $3$, for $p$ close enough to $1$ and $q>1$, uniqueness occurs but strong uniqueness does not. This is the content of Theorems 6.3 and 6.4 in \cite{GrimmettJanson}.

In the next Section, we state our main result about the strong uniqueness of the random cluster model.
The two properties of Proposition \ref{prop:inclusion_FK_measures} will be crucial in the proof of this result.

\subsection{Phase diagram of the FK model}\label{Phase diagram of the FK model}
Let $G=(V,E)$ be an infinite, locally finite, connected graph.
When $q \geq 1$, the finite dimensional FK measures $\phi^\alpha_{\Lambda,p,q}$ are strongly positively-associated for all finite subgraph $\Lambda \Subset G$ and $\alpha \in \Pi(\partial \Lambda)$, which means in particular that the \emph{FKG inequality} holds:
for any two increasing events $A$ and $B$, \[\phi_{G,p,q}^\alpha(A \cap B) \geq\phi_{G,p,q}^\alpha(A)\phi_{G,p,q}^\alpha(B). \]
The FKG inequality enables to show the existence of two particular infinite-volume measures $\phi^0_{p,q}$ and $\phi^1_{p,q}$, obtained as weak limits of the free and wired boundary conditions respectively.
One can also show that they both are invariant by automorphism of the graph.

The aforementioned Burton--Keane argument then implies that they are $\Phi-$Gibbs for any convention $\Phi$ if the graph $G$ is amenable.
Furthermore, in this case, one has the following monotonicity property:
every $\phi \in \operatorname{FK}(p,q)$ satisfies the stochastic dominations
\begin{equation}
	\phi^0_{p,q} \preceq \phi \preceq \phi^1_{p,q}.
\end{equation}
Strong uniqueness and uniqueness are therefore the same, and hold if and only if $\phi^0_{p,q}=\phi^1_{p,q}$.
It is known that for any $q \geq 1$, this fails for at most countably many edge-weights $p$ \cite[Theorem 4.63]{grimmett2006random}. 
It is then possible to show the existence of a \emph{critical parameter} $p_c(q) \in (0,1)$ for $q\geq 1$ such that every $\phi \in \operatorname{FK}(p,q)$ satisfies \[\phi(\exists\text{ an infinite cluster}) = \begin{cases}0& \text{ if }p<p_c(q)\\1& \text{ if }p>p_c(q)\end{cases}.\]
In dimension 2, the planar duality allows computing the critical point, and in particular it is shown in \cite{BeffaraDuminil} that for $G=\bbZ^2$, the self-dual point is critical, that is \[\forall q \geq 1, \ p_c(q) = \dfrac{\sqrt q}{1+\sqrt q}.\]
Moreover, we know precisely the regime where uniqueness holds.
Namely, when $q\in[1,4]$ there is a unique infinite volume measure for all values of $p$ (see \cite{DST}); when $q>4$, there is uniqueness everywhere but at $p_c(q)$ (see \cite{DGHMT}).

For $q<1$ however, the existence of the critical point is an open question, and the wider regime where uniqueness is known is obtained by the comparison inequalities.
The stochastic dominations \eqref{eq:comparison_inequalities} observed for finite graphs extend directly (we give a proof of this fact in Section \ref{Extending comparison inequalities to infinite graphs} for our enhanced comparison inequalities) to all infinite volume measures: for every FK measure $\phi$ on $G$, one has
\begin{equation}
	\PP^G_{\min(p,p')} \preceq \phi \preceq \PP^G_{\max(p,p')}
\end{equation}
For $q<1$, assuming that $p_c(q)$ exists, this implies that $\frac{p_c(q)}{p_c(q)+q(1-p_c(q))} \geq p_c(G) \geq p_c(q)$.
For $G=\bbZ^2$, the value of $p_c(G)=p_c(1)$ is known to equal 1/2 \cite{Kesten}, which gives
\[\forall q<1, \ \frac{q}{1+q} \leq p_c(q) \leq \frac12.\]
Theorem \ref{thm:new_bound_pc(q)} implies that these two putative inequalities are strict.

When the product measure $\PP^G_{\max(p,p')}$ is subcritical, one also obtains uniqueness of the FK measure:
\begin{theorem}[Theorem 5.119 in \cite{grimmett2006random}]\label{thm:uniqueness_Grimmett}
	For all $q>0$, one has strong uniqueness when $\max(p,p') < p_c(G)$ and the unique FK measure is subcritical.
\end{theorem}

The result still holds at $p_c(G)$ if $\theta(p_c):=\PP^G_{p_c(G)}(\exists \text{ an infinite cluster})$ equals $0$. We prove this without assuming that $\theta(p_c)=0$, and in fact that it is true slightly above $p_c(G)$, for graphs satisfying the assumptions of Theorem \ref{thm:new_bound_pc(q)}.

\begin{theorem}[Strong uniqueness in the subcritical regime]\label{thm:subcritical_strong_uniqueness}
	Let $G$ be a graph satisfying the hypotheses of Theorem \ref{thm:new_bound_pc(q)}. For all $q \neq1$, there exists $\delta'>0$ such that one has strong uniqueness when $\max(p,p') < p_c(G)+\delta'$ and the unique FK measure is subcritical.
\end{theorem}
The $\delta'$ in the above statement is the same as in Theorem \ref{thm:new_bound_pc(q)}.

There is no satisfactory analogue to the above theorem for the supercritical regime. In fact, a proof of uniqueness in the case of domination from below by a supercritical product measure is missing.
In arbitrary dimension, uniqueness can therefore be proved only for very large $p$.
However, for plane graphs, the duality enables to transpose the subcritical statements of Theorems \ref{thm:uniqueness_Grimmett} and \ref{thm:subcritical_strong_uniqueness} to the supercritical regime (we refer to the last Section \ref{Uniqueness in the supercritical phase through planar duality} for the definitions regarding planar duality).
Strong uniqueness on $\bbZ^2$ was therefore already known for $\min(p,p') \geq 1/2$, regardless of the value of $q>0$. As before, we slightly improve this range of parameters, for $\bbZ^2$ and more general plane graphs:
\begin{theorem}[Strong uniqueness in the supercritical regime in 2D]\label{thm:supercritical_strong_uniqueness}
	Let $G$ be as in Theorem \ref{thm:subcritical_strong_uniqueness}; we further assume that $G$ is embedded in $\bbR^2$, and that its dual $G^*$ is locally finite. For all $q\neq 1$, there exists $\delta''>0$ such that one has strong uniqueness when $\min(p,p') >  1-p_c(G^*)-\delta''$. Moreover, almost surely for the unique FK measure, there is exactly one infinite cluster.
\end{theorem}

\begin{remark}
	It follows from the sharpness of the phase transition for Bernoulli percolation \cite{DuminilTassion,BeekenkampHulshof} that we always have $1-p_c(G^*) \geq p_c(G)$ for locally finite planar quasi-transitive graphs.
\end{remark}
Obviously, for $G=\bbZ^2$, Theorems \ref{thm:subcritical_strong_uniqueness} and \ref{thm:supercritical_strong_uniqueness} are only new in the regime $q<1$ (thus $\min(p,p')=p$), but in arbitrary dimension, these might give new results even for $q>1$.
As already mentioned above, on $\bbZ^d$ for any $d \geq 3$, it is known that the set of points where uniqueness fails is at most countable for any fixed $q>1$, and included in $[p_c(q), 1-\varepsilon_{d,q}]$ for some $\varepsilon_{d,q}>0$
(see \cite{grimmett2006random} Theorems 4.63 and 5.33). However, we are unaware of any better bounds on $p_c(q)$ than those given by the comparison inequalities.
At least, our results imply that for all $q>1$, $$p_c(\bbZ^d)+\delta' \leq p_c(q) \leq q\left(q-1+\frac{1}{p_c(\bbZ^d)-\delta} \right)^{-1}.$$

\bigskip
In Section \ref{Enhanced comparison inequalities for general finite graphs}, we prove Theorem \ref{thm:enhanced_comparison_ineq}, by introducing the \emph{Glauber dynamics} of the model.
In Section \ref{Extending comparison inequalities to infinite graphs}, we extend our new comparison inequalities to the infinite-volume setting, providing a proof of Theorem \ref{thm:new_bound_pc(q)}.
Finally, Theorem \ref{thm:subcritical_strong_uniqueness} is proved in Section \ref{Uniqueness in the subcritical phase}.
Theorem \ref{thm:supercritical_strong_uniqueness} follows, by a duality argument developed in Section \ref{Uniqueness in the supercritical phase through planar duality}.

\section{The classical proof of uniqueness in the subcritical regime}
In his monograph on the random cluster model published in 2006 \cite{grimmett2006random}, Grimmett proves uniqueness for the random cluster model in the $q<1$ regime when $p'<p_c(G)$ (Theorem 5.119).
The proof works also when $q \ge 1$ (the only modification to do being to exchange the roles of $p$ and $p'$), but much stronger uniqueness results are obtained using FKG inequality, which does not hold true when $q <1$.
In this section, we briefly recall the proof of Theorem \ref{thm:uniqueness_Grimmett} due to Grimmett. Most of the notations are borrowed as is from \cite{grimmett2006random}.
Grimmett considered only weak limits with boundary condition arising from an external configuration, and the free convention at infinity for Gibbs measures. However, in this regime of subcriticality, proving uniqueness of these measures implies strong uniqueness (item (2) of Proposition \ref{prop:inclusion_FK_measures}), so we shall follow the same outline as Grimmett.
Thus, for this section, to simplify the notation, we write $\phi^{\Phi_\Lambda^f(\xi)}_{\Lambda, p,q}=:\phi^\xi_{\Lambda, p,q}$ for all $\xi \in \Omega$ and $\Lambda \Subset G$.

Before proceeding to prove the theorem, we develop some of the notations. Let $\Lambda \subset \Delta$ be finite subgraphs of $G$. Denote by $\partial \Lambda \leftrightarrow \partial \Delta$ the event that there is an open path from $\partial\Lambda$ to $\partial \Delta$.
A \textit{cutset} $S$ is a \textit{minimal} set of edges with the property that every path connecting $\Lambda$ and $\partial\Delta$ uses at least one edge from $S$. We define \textit{the interior of a cutset} $S$, denoted $\operatorname{int}(S)$, to be the set of all edges which have an end-vertex $x$ with the property that $x$ cannot be connected to $\partial \Delta$ without using an edge of $S$. It can be easily verified that the following defines a partial order on the family of cutsets, 
\begin{equation}
	S_1 \leq S_2 \iff \tilde{S_1} \subset \tilde{S_2} \quad \text{ where } \tilde{S} := S \cup \operatorname{int}(S). 
\end{equation}

The approach is somewhat similar to the disagreement percolation introduced in \cite{vdBerg} where the author proves the uniqueness of Gibbs measures for Markov fields. Here, we don't have Markov property, but we have the `free-boundary Markov property' which we adapt in the proof. This property means that conditionally on the fact that a subset of edges is closed, the connected components of the complement are independent. Clearly, this holds for all FK measures.

We prove a coupling lemma which is the main ingredient of the proof. Here we use the comparison inequalities and the `free-boundary Markov property':
\begin{lemma}
	Let $\Lambda, \Delta, \Sigma$ be boxes such that $\Lambda \subseteq \Delta \subseteq \Sigma$ and $\tau,  \xi \in \Omega$. Then there exists a probability measure $\Psi_{\Sigma}$ on $\{0,1\}^{E(\Sigma)} \times \{0,1\}^{E(\Sigma)} \times \{0,1\}^{E(\Sigma)}$ such that the following conditions are satisfied: 
	\begin{enumerate}
		\item The set of triplets $(\omega_1, \omega_2, \omega_3)$ with $\omega_2 \leq \omega_3$ and $\omega_1 \leq \omega_3$ has probability 1,
		\item The first marginal of $\Psi_{\Sigma}$ is $\phi^\xi_{\Sigma, p, q}$, the second one restricted to $E_{\Delta}$ is $\phi^\tau_{\Delta, p, q}$ and the third is the product measure $\bbP^{\Sigma}_{p'}$,
		\item Let $M$ denote the maximal cutset of $\Delta$ (with respect to $\Lambda$) all of whose edges are closed in $\omega_3$. Note that $M$ exists iff $\omega_3 \in \{ \partial\Lambda \not\leftrightarrow \partial\Delta\}$. Conditional on $M$, the marginal law of both $\{\omega_1(e): e \in \operatorname{int}(M)\}$ and $\{\omega_2(e): e \in \operatorname{int}(M)\}$ is the free measure $\phi^0_{\operatorname{int}(M), p, q}$.
	\end{enumerate}
\end{lemma}

We show how the theorem is a consequence of the above lemma:
\begin{proof}[Proof of Theorem \ref{thm:uniqueness_Grimmett} using the lemma]
	Let $A \in \mathcal{F}_{\Lambda}$; note the following consequence of Bayes' theorem:
	
	\begin{multline*}
		\Psi_{\Sigma}(\omega_1 \in A) = \Psi_{\Sigma}(\omega_1 \in A| \omega_3 \in \{ \partial\Lambda \not\leftrightarrow \partial\Delta\})\Psi_{\Sigma}(\omega_3 \in \{ \partial\Lambda \not\leftrightarrow \partial\Delta\}) + \\\Psi_{\Sigma}(\omega_1 \in A| \omega_3 \in \{ \partial\Lambda \leftrightarrow \partial\Delta\})\Psi_{\Sigma}(\omega_3 \in \{ \partial\Lambda \leftrightarrow \partial\Delta\}), 
	\end{multline*}
	
	\begin{multline*}
		\Psi_{\Sigma}(\omega_2 \in A) = \Psi_{\Sigma}(\omega_2 \in A| \omega_3 \in \{ \partial\Lambda \not\leftrightarrow \partial\Delta\})\Psi_{\Sigma}(\omega_3 \in \{ \partial\Lambda \not\leftrightarrow \partial\Delta\}) + \\\Psi_{\Sigma}(\omega_2 \in A| \omega_3 \in \{ \partial\Lambda \leftrightarrow \partial\Delta\})\Psi_{\Sigma}(\omega_3 \in \{ \partial\Lambda \leftrightarrow \partial\Delta\}).
	\end{multline*}
	
	Using part (3) of the Lemma, we see that the first terms in the above two expansions are equal (by conditioning again on $M$). Thus the difference is less than $\Psi_{\Sigma}(\omega_3 \in \{ \partial\Lambda \leftrightarrow \partial\Delta\})$. Hence, using parts (1), and (2) of the lemma about the marginals of $\Psi_{\Sigma}$ we get that, 
	\begin{equation}\label{eq:domination_difference}
		\left|\phi^\xi_{\Sigma, p, q}(A) - \phi^\tau_{\Delta, p, q}(A)\right| \leq \bbP^{\Sigma}_{p'}(\omega_3 \in \{ \partial\Lambda \leftrightarrow \partial\Delta\}).
	\end{equation}
	Now, we will use this to prove the uniqueness of the weak limit. We know that the set of weak limits is nonempty, let $\rho$ be such a measure. By definition, there exists an external configuration $\tau$ and an increasing sequence of boxes $\Delta_n$ such that $\phi^{\tau}_{\Delta_n,p,q} \implies \rho$ (weakly) as $n \to \infty$. Suppose that $\rho^{\prime} \neq \rho$ is also a weak limit, that is there is a $\xi$ and an increasing sequence of boxes, $\Sigma_n$ such that $\phi^{\xi}_{\Sigma,p,q} \implies \rho^{\prime}$ as $n \to \infty$. Let $A$ be a local event measurable with respect to some finite (fixed) box $\Lambda$ and let $m$ be such that $\Lambda \subset \Delta_m$ and $n = n_m$ be such that $\Delta_m \subset \Sigma_{n_m}$. By equation \eqref{eq:domination_difference}, we have for all $n > n_m$,
	\begin{equation*}
		\left|\phi^\xi_{\Sigma_{n}, p, q}(A) - \phi^\tau_{\Delta_m, p, q}(A)\right| \leq \bbP^{\Sigma_n}_{p'}( \partial\Lambda \leftrightarrow \partial\Delta_m).
	\end{equation*}
	Now letting $n \to \infty$ and then $m \to \infty$ gives that $\rho(A) = \rho^{\prime}(A)$ (recall that $p'$ is subcritical by hypothesis, so the right-hand side goes to 0). Since this is true for all cylinder events (which are local), we conclude that $\rho = \rho^{\prime}$, which is a contradiction. Hence $\rho$ is the unique weak limit measure. Now let us show the uniqueness of the Gibbs measure, which is an easy consequence of the above, and the subcriticality of $p'$. 
	Let $\phi$ be a Gibbs measure; then for every box $\Delta$, by the DLR equations, we have that 
	\begin{equation*}
		\phi(A \mid \mathcal{F}_{\Delta})(\xi) = \phi^\xi_{\Delta, p, q}(A) \quad \phi-\text{a.s.}
	\end{equation*}
	Thus using this and the bounded convergence theorem we get that, 
	\begin{align*}
		|\phi(A) - \rho(A)| &= \lim |\phi(\phi(A|\mathcal{F}_{\Delta_m})) - \phi^\tau_{\Delta_m, p, q}(A)|\\
		&\leq \lim \bbP^{\Delta_m}_{p'}( \partial\Lambda \leftrightarrow \partial\Delta_m) = 0.
	\end{align*}
	In the last equality, we have used the subcriticality of $p'$. This means that the set of DLR measures is singleton $\{\rho\}$.  Thus we have proved uniqueness for the $p' < p_c(G)$ regime.
\end{proof}

\begin{proof}[Sketching the proof of Lemma 2.1]
	We use an exploration method. Let $(e_l : l = 1, \cdots, L)$ be a deterministic ordering of edges in $E_{\Sigma}$. We build our configurations inward, starting from the boundary of $\Sigma$.  Let $e_{j_1}$ be the first edge incident on the boundary of $\Sigma$. We know by \eqref{eq:comparison_inequalities} that $\phi^\xi_{\Sigma, p, q} \preceq \bbP^{\Sigma}_{p'}$ and also that $\phi^\tau_{\Delta, p, q}\preceq \bbP^{\Delta}_{p'}$.
	Thus we can choose $\{0,1\}-$valued random variables $\omega_1(e_{j_1})$ and $\omega_3(e_{j_1})$ such that $\omega_1(e_{j_1})\leq \omega_3(e_{j_1})$ and they have the first and the second distributions. Set $S_1 = \{e_{j_1}\}$ and $\tilde{\xi}_1 = \xi \cup \{e_{j_1} \text{ has the state } \omega_1(e_{j_1}) \}$. So we `evolve' the boundary conditions at each step. 
	After stage $r$, we have information about $\tilde{\xi}_r$; note that by the uniformity of comparison inequalities in boundary conditions, we also have that $\phi^{\tilde{\xi}_r}_{\Sigma \setminus \{e_{j_1}, \cdots, e_{j_r}\}, p, q} \preceq \bbP^{\Sigma \setminus \{e_{j_1}, \cdots, e_{j_r}\}}_{p'}$, thus we may choose (given $e_{j_{r+1}}$) $\{0,1\}-$valued random variables $\omega_1(e_{j_{r+1}})$ and $\omega_3(e_{j_{r+1}})$ such that $\omega_1(e_{j_{r+1}})\leq \omega_3(e_{j_{r+1}})$ and they have the first and the second distributions (to observe that the joint distributions are as we want, note that putting additional boundary conditions on some edges is equivalent to conditioning the values of the configuration on those edges).
	Now we give a way to choose the sequence of edges and define the process in detail. The idea is to explore the ``open clusters of $\partial \Sigma$''. Let $e_{j_1}$ be defined as above; at step $r$, let, $S_r = \{e_{j_{s}}: s \in \{1, \cdots,r\}\}$, and $K_r = \{x \in \Sigma: x \text{ is connected to } \partial \Sigma \text{ by a path  which is open in }\omega_3\}$. Let $e_{j_{r+1}}$ be the first edge not in $S_r$ but possessing an end-vertex in $K_r$ (note that $\partial \Sigma \subset K_r$).
	Say that the process comes to a halt at stage $R$, and let $F_R$ denote the (random) set of closed edges. Note that $F_R$ consists precisely of those edges which have at least one end-vertex in $K_R$ and have been determined to be closed in $\omega_3$ (consequently in $\omega_1$).
	Now, when extending the configuration inside, the only relevant information gathered is that the edges on the inner boundary of explored edges are all closed. Consequently we can set the the remaining edges inside according to the free boundary condition on the set $F_R$, meaning they are equal in the interior.
	This is where we make use of the free-boundary Markov property.
	We remark that we can include $\omega_2$ in this exploration because of the stochastic inequality and the fact that we do not want $\omega_1$ and $\omega_2$ satisfying any other relation except for the domination, to conclude the proof.     
\end{proof}

As presented above, the proof of uniqueness for the `$p' < p_c(G)$' regime follows roughly the following idea: the random cluster measure is dominated by a subcritical Bernoulli percolation. Consequently, if we consider an FK measure on a large enough box for any fixed local event, then by subcriticality the local event is surrounded by a cutset of closed edges. Thus, the local event `experiences' the free boundary condition inside the cutset. This uses the crucial fact that we have the free-boundary Markov property for the coupling of the FK measure with the subcritical percolation that dominates it. It goes in the direction of showing that every boundary condition (and consequently every Gibbs measure) can be approximated arbitrarily closely by an FK measure with a free boundary condition, and we argue that this implies uniqueness of the infinite volume measure.

Hence, it is natural to try to get stronger comparisons with product measures, in order to expand the known uniqueness regime. In the upcoming sections, we introduce a tool called Glauber dynamics and use it later to obtain such stronger stochastic dominations, we then use essential enhancement results to show that it can strictly improve the uniqueness regime.
However, as we shall see, the coupling we get via the Glauber dynamics between FK measures and a ``smaller'' product measure does not verify the free-boundary Markov property. Therefore, we cannot imitate the above proof by exploration and a different strategy is required.

\section{Enhanced comparison inequalities for general finite graphs}\label{Enhanced comparison inequalities for general finite graphs}

Let $G=(V,E)$ be a finite graph, $\partial G \subset V$ and $\alpha \in \Pi(\partial G)$ a boundary condition. Recall that for every $e \in E$,
\begin{equation}\label{eq:cond_probas}
	\phi^\alpha_{G,p,q}\left(\omega(e)=1 \mid \omega_{\langle e\rangle}\right)= \begin{cases}p & \text { if } \omega_{\langle e\rangle} \in K^\alpha_e \\ p' & \text { if } \omega_{\langle e\rangle} \notin K^\alpha_e\end{cases}
\end{equation}
where $p'=\frac{p}{p+q(1-p)}$ and $K_e^\alpha$ is the event that the endpoints of $e$ are connected by an open path not using $e$. 
We will deduce Theorem \ref{thm:enhanced_comparison_ineq} from the two following propositions.

For all $e\in E$, a subset $\chi \subset E \setminus \{e\}$ of edges is called a \emph{cutset} for $e$ if the endpoints of $e$ are in two different connected components in $G\cup\alpha -(\chi \cup \{e\})$; equivalently, each cycle in $G\cup\alpha$ passing through $e$ contains an edge in $\chi$. Observe that if all the edges of a cutset for $e$ are closed, then the event $K_e^\alpha$ is not achieved.
In fact, the existence of a closed cutset is exactly the complementary event of $K_e^\alpha$.

\begin{proposition}\label{prop:below_enhancement}
	Let $\widehat{E} \subset E$ be such that for every edge $e\in \widehat{E}$, there exists a cutset for $e$ included in $E \setminus \widehat{E}$. Denote by $Q_{\widehat{E}}(e)$ the collection of such cutsets, and for all $e \in \widehat{E}$, let $\widehat{\varepsilon}_e:=|p'-p| \PP^G_{\max(p,p')}(\exists \chi \in Q_{\widehat{E}}(e) , \chi \text{ is closed})$.
	Then,
	\begin{enumerate}
		\item If $p<p'$, then $\PP^G_{p + \widehat{\varepsilon}_e \ind{e \in \widehat{E}}} \preceq \phi_{G,p,q}^\alpha$.
		\item If $p>p'$, then $\phi^\alpha_{G,p,q} \preceq \PP^G_{p - \widehat{\varepsilon}_e \ind{e \in \widehat{E}}}.$
	\end{enumerate}
	
\end{proposition}

\begin{proposition}\label{prop:above_enhancement}
	Let $\widetilde{E} \subset E$ be such that for every edge $e \in  \widetilde{E}$, there exists a path connecting both endpoints of $e$ included in $E \setminus \widetilde{E}$.
	Denote by $\Gamma_{\widetilde{E}}(e)$ the collection of such paths,
	and for all $e \in \widetilde{E}$, let $\widetilde{\varepsilon}_e:=|p'-p| \PP^G_{\min(p,p')}(\exists \gamma \in \Gamma_{\widetilde{E}}(e), \gamma \text{ is open})$.
	Then,
	\begin{enumerate}
		\item If $p<p'$, then $\phi_{G,p,q}^\alpha \preceq \PP^G_{p'-\widetilde{\varepsilon}_e \ind{e \in \widetilde{E}}}$.
		\item If $p>p'$, then $\PP^G_{p' + \widetilde{\varepsilon}_e \ind{e \in \widetilde{E}}} \preceq \phi_{G,p,q}^\alpha.$
	\end{enumerate}
	
\end{proposition}

\begin{remark}
	Notice that in Proposition \ref{prop:below_enhancement}, the condition on $\widehat{E}$ is satisfied if $\widehat{E}$ never contains two consecutive edges.
	In Proposition \ref{prop:above_enhancement}, the condition on $\widetilde{E}$ is satisfied if $(G, E \setminus \widetilde{E})$ is connected.
\end{remark}

Before proving Proposition \ref{prop:below_enhancement} and \ref{prop:above_enhancement}, we introduce the main tool of both proofs, that is the Glauber dynamics of the random cluster model.
To each edge $e \in E$, we independently assign an i.i.d. sequence $\left(\xi_1^e, \xi_2^e, \ldots\right)$ of exponential random variables with mean 1, and an independent i.i.d. sequence $\left(U_1^e, U_2^e, \ldots\right)$ of uniform random variables on $[0,1]$.
The idea is to update the state of an edge when its `exponential clock rings' according to its conditional distribution knowing the rest of the configuration.
For $e \in E$ and $k=1,2, \ldots$, let $\tau_k^e:=\xi_1^e+\ldots+\xi_k^e$ so that $\left(\tau_1^e, \tau_2^e, \ldots\right)$ are the jump times of a Poisson process of rate 1, and will be the update times of edge $e$.

For every $\omega \in \{0,1\}^E$, define a continuous-time $\{0,1\}^E$-valued Markov chain $(X^\omega_t)_{t\geq 0}$ with initial configuration $X^\omega_0 = \omega$.
The evolution of $X^\omega$ is governed as follows:
the state of an edge $e$ cannot change except possibly at times $\tau^e_1, \tau_2^e, \dots$ and
at time $t=\tau_k^e$ for some $k\geq 1$, let
\begin{equation}\label{eq:update}
	X^\omega_t(e) =
	\begin{cases}
		1 & \text{ if } U_k^e < p \text{ and }X^\omega_{t-} \in K^\alpha_e \\
		1 & \text{ if } U_k^e < p' \text{ and }X^\omega_{t-} \notin K^\alpha_e \\
		0 & \text{ otherwise}
	\end{cases}.
\end{equation}

It is easy to see that this defines an irreducible and aperiodic Markov chain.
According to \eqref{eq:cond_probas}, one can observe that $\phi^G_{p,q}$ is invariant for this dynamics.
Therefore, we get that $X^\omega_t$ converges in distribution towards $\phi^\alpha_{G,p,q}$ as $t$ goes to infinity. 

Note that although the dynamics is not local ---one might need to explore the configuration far away from $e$ to determine if the event $K_e$ is satisfied or not---, it is well defined because the graph is finite.
Also, since we are working with exponential clocks which are continuous, we have that almost surely, $\tau_k^e \neq \tau_j^{e^{\prime}}$ for all $j, k$ when $e \neq e^{\prime}$, so that we are never updating two edges at the same time.

\begin{proof}[Proof of Proposition \ref{prop:below_enhancement}]
	Let $\widehat{E}$ be as in Proposition \ref{prop:below_enhancement}; edges of $\widehat{E}$ will be called \emph{enhanced edges}. The other edges are called \emph{auxiliary} edges.
	
	Define two Markov processes $(X_t)_{t\geq0}$ and $(Y_t, Z_t)_{t\geq0}$ taking values in the set of configurations and of pairs of configurations respectively; $(X_t)_{t\geq0}$ will be a modified version of the Glauber dynamics of the random cluster model introduced above, while $(Y_t,Z_t)_{t\geq0}$ will contain two approximations.
	They will be coupled in such a way that $X_t$ stays a.s. above $Y_t$ and below $Z_t$ for all $t>0$ for the usual partial order on $\{0,1\}^E$.
	Let $\omega \in \{0,1\}^E$ and set $X_0 = Y_0 = Z_0 = \omega$.
	As in the Glauber dynamics, for process $(X_t)_{t\geq0}$, edges are resampled at their update times following the rule \eqref{eq:update}. The only difference is that, after each resampling of an enhanced edge, we also resample all the auxiliary edges.
	Thus, for $e \in \widehat{E}$, in addition to the $(U^e_k)_{k\geq 1}$, let $(U^e_k[1], \dots U^e_k[A])_{k \geq 1}$ be a collection of i.i.d. random variables, distributed uniformly on $[0,1]$, where $A=|E \setminus \widehat{E}|$ is the number of auxiliary edges.
	Enumerate by $e_1, \dots, e_A$ the auxiliary edges. 
	
	Therefore, for every $e \in \widehat{E}$ and $k\geq 1$, at time $\tau_k^e$, we execute the following steps:
	\begin{enumerate}
		\item[(1)] Let $X'_0(e) =
		\begin{cases}
			1 & \text{ if } U_k^e < p \text{ and }X_{\tau_k^e-} \in K^\alpha_e \\
			1 & \text{ if } U_k^e < p' \text{ and }X_{\tau_k^e-} \notin K^\alpha_e \\
			0 & \text{ otherwise}
		\end{cases}$, and $X'_0 = X_{\tau^e_k-}$ outside $e$
		
		\item[(2)] Let $X'_1(e_1) =
		\begin{cases}
			1 & \text{ if } U_k^e[1] < p \text{ and }X'_0 \in K^\alpha_{e_1} \\
			1 & \text{ if } U_k^e[1] < p' \text{ and }X'_0 \notin K^\alpha_{e_1} \\
			0 & \text{ otherwise}
		\end{cases}$, and $X'_1 = X'_0$ outside $e_1$
		
		\item[\vdots]
		\item[$(A+1)$] Let $X'_A(e_A) =
		\begin{cases}
			1 & \text{ if } U_k^e[A] < p \text{ and }X'_{A-1} \in K^\alpha_{e_A} \\
			1 & \text{ if } U_k^e[A] < p' \text{ and }X'_{A-1} \notin K^\alpha_{e_A} \\
			0 & \text{ otherwise}
		\end{cases}$, and $X'_A = X'_{A-1}$ outside $e_A$;
	\end{enumerate}
	
	finally let $X_{\tau^e_k} = X'_A$ (whereas for $e \notin \widehat{E}$, we only execute step 1 and set $X_{\tau^e_k}=X_0'$).
	Since we update only the state of the edges according to their conditional distribution, it remains true that $\phi^\alpha_{G,p,q}$ is the stationary distribution of the process $(X_t)_{t\geq0}$.
	
	Now, define the evolution of $(Y_t,Z_t)_{t \geq 0}$, and check at the same time that each update preserves the order between the three configurations. Thus, suppose that $t=\tau_k^e$ is an update time of edge $e$, and we have $Y_{t-} \leq X_{t-} \leq Z_{t-}$.
	
	If $e$ is not an enhanced edge, let
	$Y_t(e)= \begin{cases}1 & \text { if } U^e_k< \min(p,p')  \\ 0 & \text { otherwise. }\end{cases}$ and $Y_t=Y_{t-}$ outside edge $e$.
	Similarly, let $Z_t(e)= \begin{cases}1 & \text { if } U^e_k< \max(p,p')  \\ 0 & \text { otherwise. }\end{cases}$ and $Z_t=Z_{t-}$ outside edge $e$.
	The update occurs at the same time as in process $X$, using the same uniform random variable, hence we get $Y_t(e) \leq X_t(e) \leq Z_t(e)$.
	Observe that here, if we do nothing different for the enhanced edges, we will find again \eqref{eq:comparison_inequalities} by letting $t$ go to infinity.
	
	Nevertheless, if $e \in \widehat{E}$, we allow to reveal the current configuration on the auxiliary edges to determine the new state of $e$.
	The key observation is that if we find a cutset for $e$ comprising only auxiliary edges that are closed in $Z_{t-}$, then this also holds for $X_{t-}$ and $Y_{t-}$, and we know that the event $K^\alpha_e$ does not hold for $X_{t-}$.
	Thus, $e$ is opened in $X_t$ if and only if $U^e_k < p'$, so we can also use the threshold $p'$ to define $Y_t(e)$ and $Z_t(e)$ while preserving the order.
	
	Thus, if there exists $\chi \in Q_{\widehat{E}}(e)$ such that $Z_{t-}(e') = 0$ for all $e' \in \chi$, let 
	$$Y_0'(e) = Z_0'(e) =
	\begin{cases}
		1 & \text{ if } U_k^e < p' \\
		0 & \text{ otherwise}
	\end{cases};$$
	otherwise, as for non-enhanced edges, let 
	$$Y_0'(e)=\begin{cases}1 & \text { if } U^e_k< \min(p,p')  \\ 0 & \text { otherwise. }\end{cases} \text{ and } Z_0'(e)=\begin{cases}1 & \text { if } U^e_k< \max(p,p')  \\ 0 & \text { otherwise. }\end{cases}$$
	and in both cases $Y'_0=Y_{t-}, Z'_0=Z_{t-}$ outside $e$.
	
	Then, in the same way as for process $(X_t)_{t \geq 0}$, resample all the edges of $E \setminus \widehat{E}$ using the $U^e_k[i]$ and the simple rule for non-enhanced edges (that is, open the edge when the uniform random variable is less than $\min(p,p')$ for $Y_t$ or $\max(p,p')$ for $Z_t$). This defines $Y_1', \dots, Y_A'$ and $Z_1', \dots, Z_A'$, and finally we keep $Y_t=Y'_A$ and $Z_t=Z'_A$.
	In each of these operations, the order has been preserved, so in the end we get
	$Y_t \leq X_t \leq Z_t$.
	
	We conclude the proof by computing the stationary distribution of $(Y_t,Z_t)_{t \geq 0}$, and by letting $t$ tend to infinity.
	It will differ depending on whether $p<p'$ or $p'<p$. We first handle the case $p<p'$.
	Let $Y$ and $Z$ be two random configurations, of law $\PP^G_{p + \widehat{\varepsilon}_e \ind{e \in \widehat{E}}}$ and $\PP^G_{p'}$ respectively, coupled in such a way that a.s. $Y \leq Z$ (this is possible to do so because $p+\widehat\varepsilon_e \leq p'$).
	We show that if $p<p'$, the stationary distribution of $(Y_t,Z_t)_{t \geq 0}$ is the law of $(Y,Z)$.
	In light of the coupling with $(X_t)_{t \geq 0}$, which has stationary distribution $\phi^\alpha_{G,p,q}$, it implies part (1) of the Proposition (note that here we actually do not need to know the marginal law of $Z$).
	
	Let us prove that the law of $(Y,Z)$ is indeed invariant under the updates of the dynamics. Suppose that $t$ is an update time of edge $e$ (that is $t=\tau^e_k$ for some $k$) and that $(Y_{t-}, Z_{t-})$ has the same law as $(Y,Z)$.
	If $e$ is an auxiliary edge, the update only produces a resampling of the state of $e$, and sets $e$ to be open in $Y_t$ if $U_k^e <p$, and in $Z_t$ if $U_k^e <p'$, independently of everything else. Then, $Y_t$ and $Z_t$ still have the good marginals, and verify $Y_t \leq Z_t$.
	Now, if $e \in \widehat{E}$, edge $e$ is open in $Y_t$ if the uniform random variable $U^e_k$ is less than $p$, or if it is between $p$ and $p'$ and there is a cutset for $e$ of auxiliary edges all closed in $Z_{t-}$; this happens with probability $p+\widehat{\varepsilon}_e$ (here we use the fact that in $Z_{t-}$, auxiliary edges are open independently with probability $p'=\max(p,p')$).
	More directly, $e$ is open in $Z_t$ if $U^e_k <p'$, and note that we still have $Y_t(e) \leq Z_t(e)$.
	The state of the edge $e$ in $Y_t$ and $Z_t$ is therefore a measurable function of $U^e_k$ and of the configuration on auxiliary edges at time $t-$. Furthermore, the states of the auxiliary edges are also resampled during the update while preserving the ordering, with the uniform variables $U^e_k[1], U^e_k[2], \dots$, so that the states of all the edges are independent after the update (and clearly the auxiliary edges are open independently with probability $p$ in $Y_t$ and $p'$ in $Z_t$).
	It follows that the law of $(Y,Z)$ is indeed preserved by this update. 
	Therefore, in the case $p<p'$, we have identified the stationary distribution of $(Y_t,Z_t)_{t \geq 0}$, so part (1) follows.
	
	The situation is reversed when $p>p'$, but the proof is the same. This time the stationary distribution is the law of $(Y,Z)$ where $Y$ has law $\PP^G_{p'}$ and $Z$ has law $\PP^G_{p - \widehat{\varepsilon}_e \ind{e \in \widehat{E}}}$, still coupled so that a.s. $Y \leq Z$.
	The only notable difference is that after an update, an enhanced edge $e$ is closed in $Z_t$ with probability $1-p +\widehat{\varepsilon}_e$ (that is, if the uniform variable is above $p$ or between $p'$ and $p$ and the condition on the auxiliary edges holds). This is why we get $\PP^G_{p - \widehat{\varepsilon}_e \ind{e \in \widehat{E}}}$ as second marginal, and part (2) of the Proposition follows.
\end{proof}

\begin{proof}[Proof of Proposition \ref{prop:above_enhancement}]
	The proof is very similar to the previous one, so we only point out the main differences.
	Enhanced edges are now the edges in $ \widetilde{E}$, and the others are the auxiliary edges.
	The processes $(X_t)_{t \geq 0}$ and $(Y_t,Z_t)_{t \geq 0}$ are defined similarly, differing only in the update rule of the enhanced edges for $(Y_t,Z_t)_{t \geq 0}$.
	Here, it should be observed that if it happens that all the edges of a path connecting both endpoints of $e \in \widetilde{E}$ (not using $e$) are open in $Y_{t-}$, then this also holds for $X_{t-}$ and $Z_{t-}$, so in particular $K^\alpha_e$ holds for $X_{t-}$. This means that $e$ is open in $X_t$ if and only if $U^e_k<p$, so this threshold can also be used for the update of $Y_t(e)$ and $Z_t(e)$ while preserving the order.
	So in $(Y_t,Z_t)_{t \geq 0}$, when we update an enhanced edge $e$ at time $t$, we reveal the states of the auxiliary edges; if we find a path of $\Gamma_{\widetilde{E}}(e)$ whose edges are all open in $Y_{t-}$, $e$ is set to be open if the uniform random variable is less than $p$; if we do not find such a path, $e$ is updated using the same rule as for non-enhanced edges; then, in any case, all the auxiliary edges are re-sampled, in an arbitrary order.
	
	Therefore, if $p>p'$, an update at time $t$ of an enhanced edge $e$ sets $e$ to be open in $Y_t$ if the uniform random variable is less than $p'$ or if it is between $p'$ and $p$ and there is an open path of $\Gamma_{\widetilde{E}}(e)$ in $Y_{t-}$. This occurs with probability $p'+\widetilde{\varepsilon}_e$. This update sets $e$ to be open in $Z_t$ with probability $p$. We deduce that the stationary distribution of $(Y_t,Z_t)_{t \geq 0}$ has marginals $\PP^G_{p'+\widetilde{\varepsilon}_e \ind{e \in \widetilde{E}}}$ and $\PP^G_p$ (and we still have the ordering between the two coupled configurations).
	As before the situation is reversed when $p<p'$.   
\end{proof}

\begin{proof}[Proof of Theorem \ref{thm:enhanced_comparison_ineq}]
	We assume that $G \cup \alpha$ is connected, is not a tree and contains at least two distinct vertices. This last hypothesis and the connectivity implies that there is at least one edge which is not a self-loop around a vertex, so that we can find a non-empty subset $\widehat{E}$ that satisfies the hypothesis of Proposition \ref{prop:below_enhancement}.
	Then, the fact that $G\cup \alpha$ is not a tree implies that there this at least one cycle in the graph, so that we can find a non-empty subset $\widetilde{E}$ that satisfies the hypothesis of Proposition \ref{prop:above_enhancement}. Then, since $q$ is assumed to be different from 1, we have $p \neq p'$ so $\widehat{\varepsilon}_e >0$ for all $e \in \widehat{E}$ and $\widetilde{\varepsilon}_e >0$ for all $e \in \widetilde{E}$.
	Therefore, Propositions \ref{prop:below_enhancement} and \ref{prop:above_enhancement} readily imply Theorem \ref{thm:enhanced_comparison_ineq}.
\end{proof}

It is crucial to well choose the subsets $\widehat{E}$ and $\widetilde{E}$ of $E$ to get a non-zero enhancement.
As an example let us work out the above proof for finite subgraphs of $\bbZ^2$.
The simplest example of a cutset for an edge $e$ would be to take all the other edges attached to one of the endpoint of $e$; if all these edges are auxiliary, i.e. not in $\widehat{E}$, then there is a cutset for $e$ included in $E\setminus \widehat{E}$.
In fact, it is possible to partition $\bbZ^2$ with devices (of 4 edges) so that one edge of each device can be included in $\widehat{E}$, and the other edges of the device form a cutset of the first edge. This can be done by taking as devices the 4 edges attached to any even vertex; these are called ``star-devices'' (see Figure \ref{fig:devices_Z^2}).
On the other hand, the simplest path connecting both endpoints of $e$ in $\bbZ^2$ would be to take three edges that form a square (i.e. a cycle of length 4 in the graph) together with $e$. One can therefore partition $\bbZ^2$ into ``square-devices'' (see Figure \ref{fig:devices_Z^2}); from each, one edge can be included in $\widetilde{E}$, whereas the three other edges make a path connecting both endpoints of the first edge, and they are set to be auxiliary.
Notice that its easier to think about the devices as partitioning the graph, having exactly one enhanced edge per device, but in fact the devices could share auxiliary edges. Indeed, in the proofs of Proposition \ref{prop:below_enhancement} and \ref{prop:above_enhancement}, after a resampling of an enhanced edge, we resample all the auxiliary edges, so that all the auxiliary edges can be used to give a bonus to the enhanced edge. Moreover, it is not necessary that the devices cover the full graph, only that they are sufficiently dense (see in the proof of Theorem \ref{thm:new_bound_pc(q)} below).

Therefore, if $G=(V,E)$ is a finite (connected) subgraph of $\bbZ^2$, one can choose $\widehat{E}$ to be composed of exactly one edge per star-device which is entirely included in $G$, with the additional requirement to never include in $\widehat{E}$ an edge whose both endpoints are in $\partial G$.
This is done to prevent a connection between the endpoints of an enhanced edge via the boundary condition. One can construct $\widetilde{E}$ similarly with the square-devices (one even not have to worry about a possible connection via the boundary condition, since it could only help to create of path not using enhanced edges).
This construction with the devices also yields lower bounds on the $\widehat{\varepsilon}_e$ and $\widetilde{\varepsilon}_e$:
for all $e\in \widehat{E}$, there is a closed cutset for $e$ included in $E \setminus \widehat{E}$ in particular if the three auxiliary edges of the star-device of $e$ are closed, so
$$\widehat \varepsilon_e \geq |p'-p| (1-\max(p,p'))^3.$$
Similarly, we get for all $e \in \widetilde{E}$
$$\widetilde{\varepsilon}_e \geq |p'-p| \min(p,p')^3.$$

\begin{figure}[ht]
	\centering
	\includegraphics[width=0.85\linewidth]{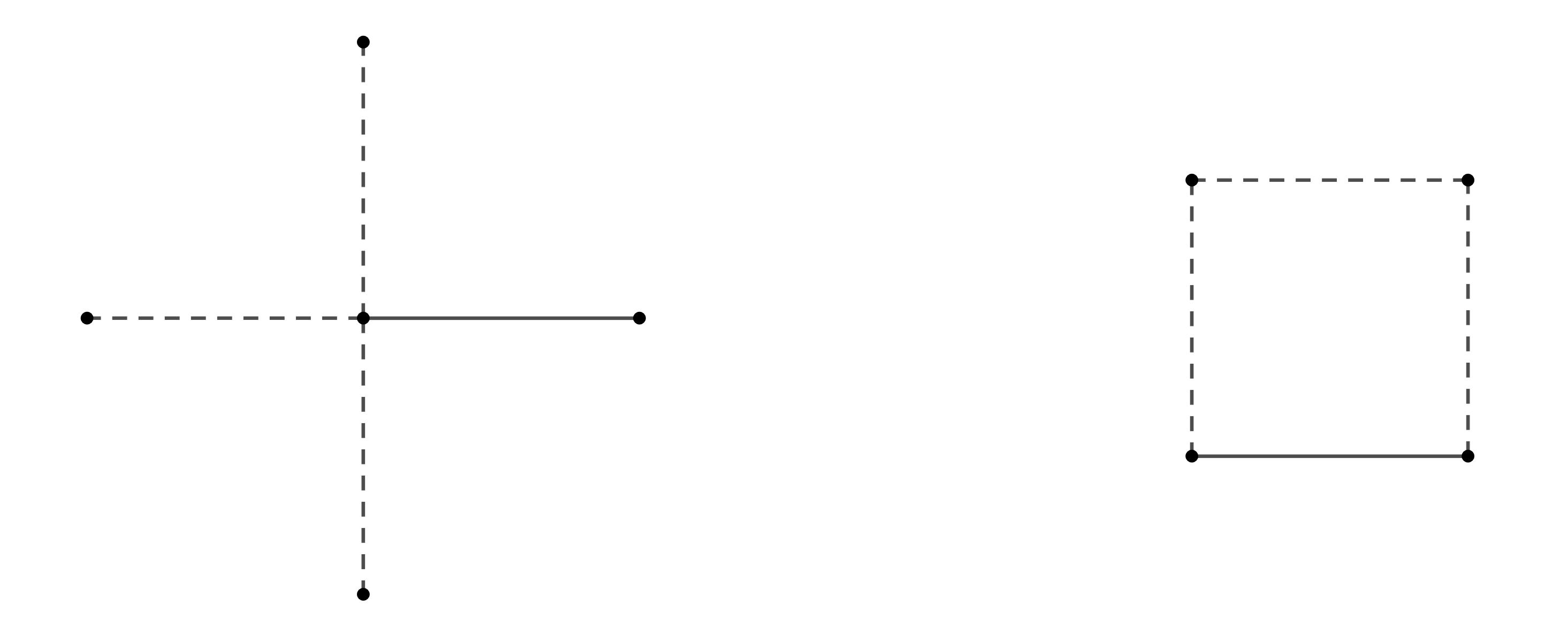}
	\caption{Left: a ``star-device''. Right: a ``a square-device''. The enhanced edge of the device is represented in unbroken line. In dotted line, we represent the other edges of the device, that ensure the presence of a cutset, or of a path connecting both endpoints, made of auxiliary edges only.}
	\label{fig:devices_Z^2}
\end{figure}

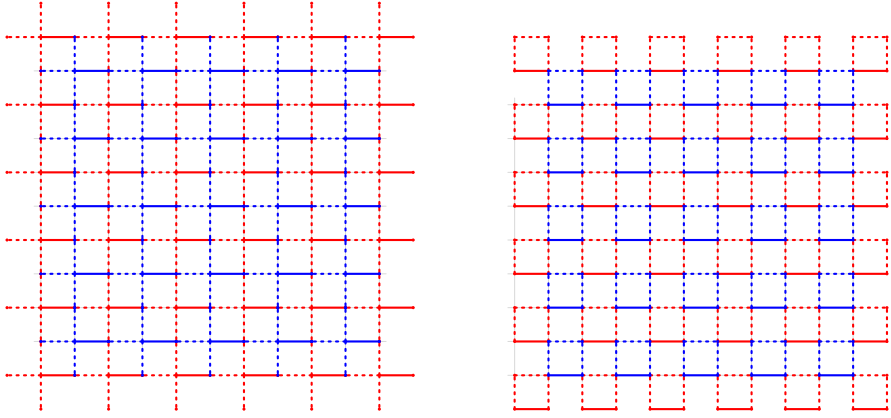
\begin{figure}[ht]
	\centering
	
	\begin{tikzpicture}[scale=0.45, line cap=round, line join=round]
		\def\W{5}\def\H{4}
		\draw[very thin, black!15] (-\W-0.2,-\H-0.2) grid (\W+0.2,\H+0.2);
		
		\newcommand{\onestar}[3]{
			\draw[dotted, thick, #3] (#1,#2) -- ++(0,1);
			\draw[dotted, thick, #3] (#1,#2) -- ++(0,-1);
			\draw[dotted, thick, #3] (#1,#2) -- ++(-1,0);
			\draw[thick, #3]         (#1,#2) -- ++(1,0);
			\fill[#3] (#1,#2) circle (1.3pt);
			\fill[#3] (#1+1,#2) circle (1.3pt);
			\fill[#3] (#1-1,#2) circle (1.3pt);
			\fill[#3] (#1,#2+1) circle (1.3pt);
			\fill[#3] (#1,#2-1) circle (1.3pt);
		}
		
		\foreach \x in {-5,-3,-1,1,3,5}{
			\foreach \y in {-5,-3,-1,1,3,5}{
				\onestar{\x}{\y}{red}
			}
		}
		\foreach \x in {-4,-2,0,2,4}{
			\foreach \y in {-4,-2,0,2,4}{
				\onestar{\x}{\y}{blue}
			}
		}
	\end{tikzpicture}
	\hspace{1cm}
	\begin{tikzpicture}[scale=0.45, line cap=round, line join=round]
		\def\W{5}\def\H{4}
		\draw[very thin, black!15] (-\W-0.2,-\H-0.2) grid (\W+0.2,\H+0.2);
		
		\newcommand{\onesq}[3]{
			\draw[thick,#3] (#1,#2) -- (#1+1,#2);
			\draw[dotted,thick,#3] (#1,#2) -- (#1,#2+1);
			\draw[dotted,thick,#3] (#1+1,#2) -- (#1+1,#2+1);
			\draw[dotted,thick,#3] (#1,#2+1) -- (#1+1,#2+1);
			\fill[#3] (#1,#2) circle (1.3pt);
			\fill[#3] (#1+1,#2) circle (1.3pt);
			\fill[#3] (#1,#2+1) circle (1.3pt);
			\fill[#3] (#1+1,#2+1) circle (1.3pt);
		}
		
		\foreach \x in {-5,-3,-1,1,3,5}{
			\foreach \y in {-5,-3,-1,1,3,5}{
				\onesq{\x}{\y}{red}
			}
		}
		\foreach \x in {-4,-2,0,2,4}{
			\foreach \y in {-4,-2,0,2,4}{
				\onesq{\x}{\y}{blue}
			}
		}
	\end{tikzpicture}
	
	\caption{Left: covering $\mathbb Z^2$ with star-devices, Right: covering $\mathbb Z^2$ with square-devices.}
	\label{divice_cover_Z^2}
\end{figure}

This devices construction can be generalized to a large class of graphs.
In the next section, we build on this to extend the new comparison inequalities to the infinite-volume setting and prove Theorem \ref{thm:new_bound_pc(q)}.

\section{Extending comparison inequalities to infinite graphs}\label{Extending comparison inequalities to infinite graphs}
Let $G=(V,E)$ be an infinite, locally finite, connected graph.
As in the statement of Theorem \ref{thm:new_bound_pc(q)}, we suppose that $G$ contains edges which are on a doubly-infinite path. These edges are called \emph{pivotal}.
Let $\Delta:= \max_{v \in V} \deg(v)$ be the maximal degree of $G$ and $L:= \min_{\gamma \text{ cycle with pivotal edge}}|\gamma|$ be the minimal length of a cycle containing a pivotal edge.
In the following, for any finite subset of vertices $\Lambda \Subset V$, we will identify $\Lambda$ with the induced subgraph of $G$, and consider the boundary of $\Lambda$ to be
$$\partial \Lambda=\{x \in \Lambda \mid \exists y \in \Lambda^c, \{x,y\} \in E\}.$$
Let us start by properly define the devices.
\begin{definition}
	A \emph{device} in $G$ is a finite subset of $E$ that contains a pivotal edge. We call \emph{star-devices} the devices which are formed by all the edges attached to some vertex of $G$, and \emph{cycle-devices}\footnote{This terminology replaces the less general notion of square-devices of the previous Section.} the devices formed by a cycle of length $L$ in $G$.
\end{definition}
The restriction on the length of the cycle-devices is added to avoid considering cycles of arbitrary large length around a pivotal edge; in fact, this would cancel the enhancement (see Proposition \ref{prop:comparison_inequalities_infinite_volume} below).
Thus, $\Delta$ is the maximal size of a star-device, and $L$ is the the size of cycle-devices.
We suppose that both $\Delta$ and $L$ are finite (this is in particular the case if $G$ satisfies the hypothesis of Theorem \ref{thm:new_bound_pc(q)}, i.e. if $G$ is quasi-transitive and contains pivotal edges which are located on a cycle).
One can therefore prove the following straightforward consequence of Propositions \ref{prop:below_enhancement} and \ref{prop:above_enhancement} for infinite graphs:

\begin{proposition}\label{prop:comparison_inequalities_infinite_volume}
	Suppose that there exists $\widehat{E} \subset E$ such that for all $e \in \widehat{E}$, there is a star-device containing $e$ of which all the other edges are in $E\setminus \widehat{E}$.
	Similarly, suppose that there exists $\widetilde{E} \subset E$ such that for all $e \in \widetilde{E}$, there is a cycle-device containing $e$ of which all the other edges are in $E\setminus \widetilde{E}$.
	Let $\widehat{\varepsilon}:=|p'-p| (1-\max(p,p'))^{\Delta-1}$ and $\widetilde{\varepsilon}:=|p'-p| \min(p,p')^{L-1}$. Then, for any FK measure $\phi$ on $G$,
	\begin{enumerate}
		\item If $p<p'$, one has $\PP^G_{p+\widehat{\varepsilon} \ind{e \in \widehat{E}}} \preceq \phi \preceq \PP^G_{p'-\widetilde{\varepsilon} \ind{e \in \widetilde{E}}}$.
		\item If $p>p'$, one has $\PP^G_{p'+\widetilde{\varepsilon} \ind{e \in \widetilde{E}}} \preceq \phi \preceq \PP^G_{p-\widehat{\varepsilon} \ind{e \in \widehat{E}}}$.
	\end{enumerate}
\end{proposition}

\begin{proof}
	We do the proof only in the case $p<p'$, the proof being the same in the other case.
	Let $\phi$ be an FK measure on $G$, and first assume that it is a limit measure.
	This means that there exists an increasing sequence $\Lambda_n \uparrow G$ and a sequence $\alpha_n \in \Pi(\partial \Lambda_n)$ such that
	$$\phi= \lim_{n \to \infty} \phi^{\alpha_n}_{\Lambda_n, p,q}.$$
	For all $n\geq 0$, we take $\widehat{E}(\Lambda_n)$ to be the set of edges of $\widehat{E}$ whose associated star-device (i.e. the one included in $\{e\} \cup E\setminus\widehat{E}$, which exists by hypothesis) is entirely included in $\Lambda_n$, and whose endpoints are not both in $\partial \Lambda_n$. Define $\widetilde{E}(\Lambda_n)$ similarly (we can even relax the assumption regarding $\partial \Lambda$).
	For all $e \in \widehat{E}(\Lambda_n)$, if all the edges except $e$ of the star-device associated with $e$ are closed, it yields a closed cutset for $e$ included in $E(\Lambda_n)\setminus\widehat{E}(\Lambda_n)$; therefore, $\widehat{\varepsilon}_e$ defined in Proposition \ref{prop:below_enhancement} satisfies
	$$\widehat{\varepsilon}_e \geq |p'-p| (1-\max(p,p'))^{\Delta-1}=\widehat{\varepsilon},$$
	so Proposition \ref{prop:below_enhancement} implies that
	$$\PP^{\Lambda_n}_{p+\widehat{\varepsilon} \ind{e \in \widehat{E}(\Lambda_n)}} \preceq \phi^{\alpha_n}_{\Lambda_n, p,q}.$$
	Similarly, for all $e\in \widetilde{E}(\Lambda_n)$, if all the edges except $e$ of the cycle-device associated with $e$ are open, it yields an open path connecting both endpoints of $e$ (without using edge $e$) included in $E(\Lambda_n)\setminus\widetilde{E}(\Lambda_n)$; therefore, we have $$\widetilde{\varepsilon}_e \geq |p'-p| \min(p,p')^{L-1}=\widetilde{\varepsilon}$$ and Proposition \ref{prop:above_enhancement} gives
	$$\phi^{\alpha_n}_{\Lambda_n, p,q} \preceq \PP^{\Lambda_n}_{p'-\widetilde{\varepsilon} \ind{e \in \widetilde{E}(\Lambda_n)}}.$$
	Letting $n$ goes to infinity, we obtain $\PP^G_{p+\widehat{\varepsilon} \ind{e \in \widehat{E}}} \preceq \phi \preceq \PP^G_{p'-\widetilde{\varepsilon} \ind{e \in \widetilde{E}}}$.
	
	The proof for FK Gibbs measures (for any convention) can be done directly from the DLR equations, considering large box containing the support of a local test event and integrating over the boundary condition, or simply by noticing that FK Gibbs measures are mixtures of FK limit measures: this is a direct consequence of Prop \ref{prop:inclusion_FK_measures}, item (1).
\end{proof}

Note that in the above proof, we never use the fact that the devices contain pivotal edges. Nevertheless, this will be crucial for the following proof of Theorem \ref{thm:new_bound_pc(q)}.
We claim that a construction of subsets $\widehat{E}$ and $\widetilde{E}$ satisfying the hypotheses of Proposition \ref{prop:comparison_inequalities_infinite_volume} can be done for all quasi-transitive graphs that contain pivotal edges which belong to a cycle. The quasi-transitivity assumption should not be necessary, but it covers most of the interesting graphs on which the argument would work.
Note also that for a given explicit graph, it should not be hard to find appropriate subsets $\widehat{E}$ and $\widetilde{E}$. In the following proof, we present a generic construction, that might not give the optimal enhancement. We then apply the essential enhancement theory of Aizenman and Grimmett \cite{AizenmanGrimmett} to deduce Theorem \ref{thm:new_bound_pc(q)}.
\begin{proof}[Proof of Theorem \ref{thm:new_bound_pc(q)}]
	Suppose that $G$ is quasi-transitive and that it contains a pivotal edge which is located on a cycle; by quasi-transitivity, there are in fact infinitely many such edges. In particular, we have that $L$ is finite, and $G$ has bounded degree (because it is locally finite), so $\Delta$ is finite too.
	The construction of $\widehat{E}$ and $\widetilde{E}$ is done with the same procedure. Thus, in the following, a device refers to a star- or a cycle-device, depending on whether one aims to construct $\widehat{E}$ or $\widetilde{E}$.
	For all $e \in E$ and for all $k \geq 1$, let $B_k(e)$ be the set of edges at distance less than $k$ from $e$.
	Let $K$ be an integer large enough so that for all $e\in E$, $B_K(e)$ contains a device. Such a finite $K$ exists because the graph is quasi-transitive.
	Fix $o \in E$, and consider a device included in $B_K(o)$, and choose arbitrarily one of its pivotal edges to be included in the set of enhanced edges (either $\widehat{E}$ or $\widetilde{E}$). Then, until it is no longer possible, consider at each step a new device included in $B_K(o)$ disjoint from all the previously discovered devices; choose again one of its edges to be an enhanced edge.
	Afterwards, consider a new device included in $B_{2K}(o)$ which is not authorized to intersect the previously discovered devices but it can intersect $B_K(o)$. Do this until it is no longer possible, and then repeat the same procedure for $B_{3K}(o)$ etc.
	
	This provides a construction of $\widehat{E}$ and $\widetilde{E}$ satisfying the assumption of Proposition \ref{prop:comparison_inequalities_infinite_volume}. Let $\phi$ be an FK measure and consider the stochastic domination for $\phi$ given by Proposition \ref{prop:comparison_inequalities_infinite_volume}:
	$$\phi \preceq \PP^G_{\max(p,p')-\varepsilon \ind{e \in E'}},$$ taking $\varepsilon = \widehat{\varepsilon}$ and $E'=\widetilde{E}$ if $p>p'$, and $\varepsilon = \widetilde{\varepsilon}$ and $E'=\widehat{E}$ if $p<p'$. Each edge $e \in E$ is at distance bounded by $K$ plus the maximal diameter of a device from the set $E'$ of enhanced edges. Indeed, there is at least one edge in $B_K(e)$ belonging to a device discovered in the construction of $E'$; otherwise, we could have discover a new device included in $B_K(e)$, disjoint from the other devices of the exploration, and add one of its edges to $E'$.
	Since the edges of $E'$ have been chosen so that they are all pivotal,
	the enhancement is necessarily ``essential'' in the sense of \cite{AizenmanGrimmett} --- meaning that for any enhanced edge, there exists a configuration such that there is no doubly-infinite path if the edge is closed, but such a path exists if the edge is open.
	According to \cite{AizenmanGrimmett}, this (together with the fact that each edge of the graph is at bounded distance from an enhanced edge)
	is sufficient to guarantee that the enhancement strictly increases the critical parameter (here we need to have $0<p_c(G)<1$). More precisely,
	for all $q \neq 1$, there exists $\delta'(q)>0$ such that if $p$ satisfies $\max(p,p') < p_c(G)+ \delta'(q)$, then the percolation associated with the product measure $\PP^G_{\max(p,p')-\varepsilon \ind{e \in E'}}$ is subcritical.
	This yields item (1) of the theorem, and (2) follows from the same argument with the stochastic domination from below.
\end{proof}

\section{Proofs of strong uniqueness}
\subsection{Uniqueness in the subcritical phase}\label{Uniqueness in the subcritical phase}
In this section, we show strong uniqueness in the regime where the random cluster model is dominated by a subcritical percolation. It differs from the classical proofs of uniqueness as it relies on the structure of the set of Gibbs measures (for a fixed convention). In particular, we do not derive strong uniqueness from uniqueness of the limit measure, as classically done in the regime $q>1$.
Instead, we use the fact that extremal Gibbs measures can be obtained as weak limits along any sequence of boxes, taking as boundary condition a typical realization of the measure. This was mentioned earlier, in item (1) of Proposition \ref{prop:inclusion_FK_measures}, which is stated in \cite{friedli_velenik_2017} (Theorem 6.63). We reproduce the argument in the proof of Theorem \ref{thm:domination_gives_uniqueness} stated below, and make an appropriate choice of boxes to prove uniqueness. Theorem \ref{thm:domination_gives_uniqueness} appears to be a general statement regarding uniqueness, when all Gibbs measures for a given specification are dominated by a subcritical percolation.
Theorem \ref{thm:subcritical_strong_uniqueness} is then a direct consequence of this result plus item (2) of Proposition \ref{prop:inclusion_FK_measures}.

\begin{theorem}[Uniqueness through domination by a subcritical percolation]\label{thm:domination_gives_uniqueness}
	Fix a convention $\Phi$ and parameters $p$ and $q$.
	Assume that there exists a probability measure $\PP$ on $\{0,1\}^E$ such that $\bbP(\exists \text{ an infinite cluster)}=0$ which dominates every $\Phi-$Gibbs measure, that is for all $\phi \in \G^\Phi_{p,q}$, $\phi \preceq \bbP$. Then we have uniqueness of the Gibbs measure, that is, $|\G^\Phi_{p,q}| = 1$. 
\end{theorem}

\begin{proof}
	Let $\mu_1, \mu_2 \in \operatorname{ex} \G^\Phi_{p,q}$ be two extremal $\Phi$-Gibbs measures. We conclude by showing that $\mu_1 = \mu_2$, this would mean that $|\operatorname{ex}\G^\Phi_{p,q}| =1$ and consequently $|\G^\Phi_{p,q}| = 1$. Now let $\PP$ be as in the statement of the theorem. By stochastic domination, it is possible to find a coupling $(\omega_1, \omega _2, \omega _3)$ such that $\omega_1 \sim \mu_1, \omega_2 \sim \mu_2$ and $\omega_3 \sim \PP$, satisfying almost surely $\omega_1 \leq \omega_3$ and $\omega_2 \leq \omega_3$.
	Let $\tilde{\Omega}$ denote the set of configurations such that we can find an increasing sequence $S_n = S_n(\omega_3) \subset \bbZ^2$ of connected subgraphs of $G$ such that $S_n \uparrow G$ and $\omega_3|_{\partial S_n} = 0$. By subcriticality of $\PP$, it follows that $\PP(\tilde{\Omega}) = 1$ (and consequently, considering our coupling, $\mu_i(\tilde{\Omega}) = 1$ as well). Let $C$ be a cylindrical event. By uniqueness of conditional expectation and using the fact that $\mu_i(\cdot \mid \mathcal{F}_{S_n(\omega_3)^c})(\omega_i)$ is the free boundary random cluster measure in the interior of $S_n(\omega_3)$, we have for almost every tuple $(\omega_1, \omega_2, \omega_3)$, for large enough $n$, that: 
	$$\mu_1(C | \mathcal{F}_{S_n(\omega_3)^c})(\omega_1) =\mu_2(C | \mathcal{F}_{S_n(\omega_3)^c})(\omega_2), $$
	because $C$ is measurable with respect to the interior of $S_n(\omega_3)$ for large enough $n$. Now we use the backward martingale convergence theorem along with the fact $S_n({\omega_3})^{\circ} \uparrow \bbZ^2$ almost surely, to conclude that almost surely, $$\mu_i(C | \mathcal{F}_{S_n(\omega_3)^c})(\omega_i) \xrightarrow{n \to \infty}  \mu_i(C|\mathcal{F}_{\infty})(\omega_i) = \mu_i(C)\quad  i = 1,2. $$In the last equality we have used the fact that extremal measures are tail trivial (see \cite{friedli_velenik_2017} Theorem 6.58). This means that $\mu_1, \mu_2$ agree on cylindrical events and are thus equal.  
\end{proof}

\begin{proof}[Proof of Theorem \ref{thm:subcritical_strong_uniqueness}]
	From Theorem \ref{thm:new_bound_pc(q)}, we now that every $\phi \in \operatorname{FK}(p,q)$ is dominated by some subcritical percolation $\PP$ when $\max(p,p') < p_c(G)+\delta'.$
	This implies that $\phi-$a.s. there is no infinite cluster, so by Proposition \ref{prop:inclusion_FK_measures} item (2), it follows that $\phi \in \G_{p,q}^{\Phi^f}$, so $\operatorname{FK}(p,q) \subset \G_{p,q}^{\Phi^f}$.
	Then, in particular the domination by $\PP$ holds for every $\phi \in \G_{p,q}^{\Phi^f}$, so according to Theorem \ref{thm:domination_gives_uniqueness}, we have $|\G_{p,q}^{\Phi^f}|=1$.
	We conclude that $|\operatorname{FK}(p,q)|=1$ for this range of parameters.
\end{proof}

\subsection{Uniqueness in the supercritical phase through planar duality}\label{Uniqueness in the supercritical phase through planar duality}
We start by recalling some definitions about plane graphs and their duality.
A graph is said to be \emph{plane} if it is\footnote{Note the subtle difference with a planar graph, which \emph{can} be embedded in $\bbR^2$. We choose to consider plane graphs in order to have uniqueness of the dual graph.} embedded in the plane $\bbR^2$ such that two edges never intersect, except at a common endpoint. Let $G=(V,E)$ be a plane graph. The \emph{faces} of $G$ are defined as the connected component of the complementary of the graph when embedded in $\bbR^2$.
The \emph{dual graph} of $G$, denoted $G^*=(V^*,E^*)$, is defined as follows:
place a dual vertex at each face of $G$ (including possibly infinite faces), and for each edge $e\in E$, put a dual edge between two dual vertices whose corresponding faces of $G$ are separated by the edge $e$.
Obviously, there is a canonical bijection between $E$ and $E^*$.
Therefore, for each configuration $\omega \in \{0,1\}^E$, one can associate a dual configuration $\omega^* \in \{0,1\}^{E^*}$, by declaring as open a dual edge if and only if its corresponding (primal) edge is closed.
It is straightforward from this definition that if $\omega$ is distributed according to $\bbP^G_p$, then its dual configuration $\omega^*$ has law $\bbP^{G^*}_{1-p}$. This of course extends to inhomogeneous Bernoulli percolation measures.
The situation is more intricate for FK measures, but the FK model still have a nice feature with respect to the planar duality.

Let $G=(V,E)$ be a finite plane graph, and take $\partial G= \emptyset$, so that we do not have a boundary condition (thus we also drop it from the notation). It classically follows (see \cite{grimmett2006random} Section 6.1) from the Euler's formula for plane graphs that 
$$\forall \omega \in \{0,1\}^E, \ \phi_{G, p,q}(\omega)= \phi_{G^*,p^*,q}(\omega^*),$$
where $p^*$ is defined by the following relation:
$$\dfrac{p}{1-p}\dfrac{p^*}{1-p^*}=q.$$
Note that $p^*$ of course depends on $p$ and also on $q$, but we drop it from the notation as we did for $p'$.
In fact, there is a strong link between $p'$ and $p^*$, as we can observe that
\begin{equation*}
	p'=\dfrac{p}{p+(1-p)q}=\dfrac{1}{1+(1-p)q/p}=\dfrac{1}{1+p^*/(1-p^*)}=1-p^*.
\end{equation*}
Similarly, we have $p'^*=1-p$.

Considering the duality relation for FK measures with a boundary condition is more challenging.
Suppose now that $G$ is an infinite plane graph, and take a finite subgraph $\Lambda \Subset G$. As before, $\partial \Lambda$ is composed of the vertices of $\Lambda$ having a neighbor outside. Intuitively, one may believe that the FK measure on $\Lambda$ with some boundary condition admits a dual measure on $\Lambda^*$ with a dual boundary condition.
If $\alpha \in \Pi(\partial \Lambda)$ is a partition, there is no obvious way to define its dual.
Even if the boundary condition comes from an external configuration, ---in this case we can define the dual external configuration--- the convention at infinity that we choose affect the measure
and it is not possible in general to define the dual of a convention.
The favorable case is when the external configuration $\xi$ contains at most one infinite cluster: thus, the convention has no effect and we have for any pair of conventions $\Phi,\Phi'$, $$\phi^{\Phi_\Lambda(\xi)}_{\Lambda,p,q}=\phi^{\Phi'_\Lambda(\xi)}_{\Lambda,p,q},$$ that we will write for short $\phi^\xi_{\Lambda,p,q}$.
In this case, we have the following duality relation:
\begin{equation}\label{eq:duality_relation_FK}
	\forall \omega \in E(\Lambda), \ \phi^\xi_{\Lambda, p,q}(\omega)=\phi^{\xi^*}_{\Lambda^*, p^*, q}(\omega^*).
\end{equation}
Indeed, assuming that $\xi$ has at most one infinite cluster, it follows that the partition induced by $\xi$ at the boundary of $\Lambda$ can be determined by revealing $\xi$ on a finite box $\Delta \supset \Lambda$, as otherwise it would means that two vertices of $\partial \Lambda$ are connected by a ``path'' which is not contained in any finite box, i.e. they belong to two disjoint infinite clusters which are merged at infinity by the convention. 
It follows that we have the duality relations for the FK measures in $\Delta$ (and $\Delta^*$) without boundary condition. The relation is then preserved by conditioning the measures on $\Delta \setminus\Lambda$ (resp. $\Delta^* \setminus \Lambda^*$) to coincide with $\xi$ (resp. $\xi^*$), and this gives the same measure on $\{0,1\}^{E(\Lambda)}$ (resp. $\{0,1\}^{E(\Lambda^*)}$) as $\phi^\xi_{\Lambda, p,q}$ (resp. $\phi^{\xi^*}_{\Lambda^*, p^*,q}$).

We can now give the duality argument that allows to get strong uniqueness in the supercritical phase.

\begin{proof}[Proof of Theorem \ref{thm:supercritical_strong_uniqueness}]
	Let $\delta''>0$ and take parameters $p \in [0,1]$ and $q\neq1$ so that $\min(p,p')>1-p_c(G^*)-\delta''$. Observe that this condition is equivalent to $\max(p^*, p'^*) < p_c(G^*)+\delta''$. Let $\phi \in \operatorname{FK}(p,q)$ be an FK measure on $G$; we first prove that $\phi$ is supported on configurations with exactly one infinite cluster, providing $\delta''$ is small enough.
	We know from Proposition \ref{prop:comparison_inequalities_infinite_volume} and from our devices construction (Proof of Theorem \ref{thm:new_bound_pc(q)}) that there exists a subset $E' \subset E$ and $\varepsilon>0$ such that 
	$$\PP^G_{\min(p,p')+\varepsilon \ind{e \in E'}} \preceq \phi,$$
	and that this product measure is supercritical, so is $\phi$, if $\min(p,p')$ is above or slightly below $p_c(G)$. This is the case if $\delta''$ is small enough, as $1-p_c(G^*) \geq p_c(G)$, so in this case there is $\phi-$a.s. at least one infinite cluster.
	
	Call $\phi^*$ the dual law of $\phi$, that is, the distribution of $\omega^*$ when $\omega$ has law $\phi$.
	The above stochastic domination admits a dual version, namely
	$$\phi^* \preceq \PP^{G^*}_{1-\min(p,p')-\varepsilon \ind{e \in E'}} = \PP^{G^*}_{\max(p^*, p'^*)-\varepsilon \ind{e \in E'}}.$$
	Again by essential enhancements, applied on the dual graph $G^*$, this product measure is subcritical if $\max(p^*,p'^*)$ is below or slightly above $p_c(G^*)$, i.e. if $\delta''$ is small enough.
	Therefore, $\omega^*$ has a.s. no infinite cluster, so $\omega$ has a.s. at most one infinite cluster.
	(Note that the subcriticality of this measure does \emph{not} follow from supercriticality of its dual $\PP^G_{\min(p,p')+\varepsilon \ind{e \in E'}}$ on $G$, that we have already obtained in the proof of Theorem \ref{thm:new_bound_pc(q)} by essential enhancements; indeed, we have no guarantee that this measure on $G$ has a.s. exactly one infinite cluster!)
	
	Since we now know that there is $\phi-$a.s. exactly one infinite cluster, for $\delta''>0$ small enough, we can now argue as in the proof of Theorem \ref{thm:subcritical_strong_uniqueness}, and only prove uniqueness of the Gibbs measure for a single fixed convention; according to Proposition \ref{prop:inclusion_FK_measures} item (2), it yields strong uniqueness.
	In fact, for proving strong uniqueness, we only need the fact that there is at most one infinite cluster; in this case, the convention has no effect so for all $\Lambda \Subset G$ and for any pair of conventions $\Phi,\Phi'$, we have $\phi^{\Phi_\Lambda(\xi)}_{\Lambda,p,q}=\phi^{\Phi'_\Lambda(\xi)}_{\Lambda,p,q}=\phi^\xi_{\Lambda,p,q}$.
	
	Let $\phi \in \operatorname{ex}\G^{\Phi^f}_{p,q}$. By extremality of $\phi$, we have
	$$\phi=\lim_{\Lambda \uparrow G} \phi^\xi_{\Lambda,p,q} \text{ for $\phi-$a.e. } \xi$$ (Proposition \ref{prop:inclusion_FK_measures} item (1)).
	Now, by Theorem \ref{thm:subcritical_strong_uniqueness} applied to $G^*$, we have strong uniqueness for the FK model with parameters $p^*$ and $q$ on $G^*$, if $\delta''$ is small enough. Denote by $\phi^*$ the unique FK measure on $G^*$ with these parameters.
	Since it is unique it is necessarily extremal, so we have 
	$$\lim_{\Lambda^* \uparrow G^*} \phi^{\xi^*}_{\Lambda^*, p^*, q} = \phi^* \text{ for $\phi^*-$a.e. } \xi^*.$$
	By applying the duality relation \eqref{eq:duality_relation_FK} to all finite subgraphs $\Lambda \Subset G$, it follows that $\phi$ is the (unique) dual measure of $\phi^*$, that is the distribution of $\omega$ when $\omega^*$ has law $\phi^*$.
	Therefore, one has $|\operatorname{ex} \G^{\Phi^f}_{p,q}|=1$, so $|\G^{\Phi^f}_{p,q}|=1$ and we get strong uniqueness.
\end{proof}

\section{Open problems}
We conclude the paper with a few open questions. 
In Theorem \ref{thm:new_bound_pc(q)}, we have extended the known regimes of subcriticality and supercriticality of the FK percolation under mild assumptions on the graph. However, the reader will have noticed a major difference between these two regimes with regard to the question of the uniqueness of the Gibbs measure.
In the subcritical case, we do not need further assumption to derive uniqueness (Theorem \ref{thm:subcritical_strong_uniqueness}) while in the supercritical case, we get the result only in dimension 2 through planar duality (Theorem \ref{thm:supercritical_strong_uniqueness}).
The point is that we do not have a supercritical equivalent of Theorem \ref{thm:domination_gives_uniqueness}, which affirms that domination by a subcritical percolation implies uniqueness.
Therefore, we raise the following question:
\begin{question}
	Suppose that there exists $\PP$ a probability measure on $\{0,1\}^E$ such that for all $\phi \in \operatorname{FK}(p,q)$, $\PP \preceq \phi$ and $\PP(\exists \text{ an infinite cluster})=1$. Is is true that we have $|\operatorname{FK}(p,q)|=1$?
\end{question}

This question is not specific to the random cluster model (just like Theorem \ref{thm:domination_gives_uniqueness}).
In case of a positive answer, we can bypass the duality argument and thus remove the planar hypothesis on the graph in Theorem \ref{thm:supercritical_strong_uniqueness}.

\bigskip
For a given FK measure $\phi$, the covariance function is defined by
$$\forall e,f \in E, \ \ \langle \omega_e;\omega_f\rangle_\phi:= \phi[ \omega_e \omega_f]-\phi[\omega_e]\phi[\omega_f]$$
where $\phi[ \cdot]$ denotes the expectation with respect to $\phi$.
The covariance function is believed to decay exponentially fast with the distance between $e$ and $f$ outside the critical point.
To the best of our knowledge, this has only been proved for $q\geq 1$, in the subcritical phase: this is for example a consequence of a result of Harel and Spinka \cite{HarelSpinka}; they shown the stronger fact that the model can be realised by a factor of i.i.d. process, with exponential decay for the tail of the radius of the factor. Their proofs strongly relies on monotonicity and therefore cannot be applied to the $q<1$ regime.
By self-duality, on $\bbZ^2$, it also holds in the supercritical phase but in higher dimension, like for the question of uniqueness, it is only known for $p$ close enough to 1.
However, for $q<1$, the question is entirely open, even in the cases where we have a comparison with a sub- or super-critical Bernoulli percolation.

\begin{question}
	Do we have exponential decay of the covariance function when $\max(p,p') < p_c(G)+\delta'$ or $\min(p,p')>p_c(G)-\delta$?
\end{question}
Note that in the case of domination by a subcritical percolation, we immediately get the exponential decay of the connectivity function, that is, the probability of two vertices being connected under $\phi_{p,q}$. However, it is not clear that this implies exponential decay of the covariance.
Much more is known for $q\geq 1$, since we have in fact exponential decay of the connectivity function in the whole subcritical phase \cite{DRT}.

\bigskip
We now consider the question of the uniqueness of the infinite cluster on $\bbZ^d$ (on non-amenable graphs, it is easy to get FK measures with more than one infinite cluster, even notably for $q=1$). This is linked with the existence of non-translation-invariant FK measures.
Indeed, we have already mentioned the result of Burton and Keane \cite{BurtonKeane}, which guarantees that on $\bbZ^d$ (as well as any other amenable lattice), translation-invariant (TI) FK measures cannot have more than one infinite cluster.
In dimension 2, it is widely believed that non-TI measures do not exist, but it is in general hard to exclude formally their existence. The celebrated Aizenman--Higuchi theorem for the 2D Ising model has been extended in \cite{CDIV}, where it is proved that for all $q\in \bbN^*$ the $q$-state Potts model admits exactly $q$ extremal Gibbs measures in the supercritical phase, implying the non-existence of non-TI measures. But all these measures correspond to the same FK measure, and furthermore, we have already observed that the coexistence of different FK measures can occur only at $p_c(q)$ (if $q\geq1)$.
For $q\in[1,4]$, continuity of the phase transition \cite{DST} implies that there is a unique FK (or Potts) measure at the critical point, and this measure is translation-invariant, but the structure of the set of Gibbs measures is not fully understood for $q>4$.

However, in dimension $d \geq 3$, it is known that at $p=p_c(q)$ for large enough $q$, there exist non-TI FK measures (see \cite{CernyKotecky} or Theorem 7.35 in \cite{grimmett2006random}).
The construction in \cite{CernyKotecky} produces an extremal non-TI FK measure that has a.s. one infinite cluster. It would be surprising if a measure in the $q\geq 1$ regime could produce more than one infinite cluster, but it appears that this has not been formally ruled out.
Once again, the situation is much less well understood in the regime $q \in (0,1)$, where nothing is known about non-TI FK measures (beyond the trivial fact that they do not exist in the uniqueness phase).

\begin{question}
	Is it true that for $q>0$ and $p\in[0,1]$, every $\phi \in \operatorname{FK}(p,q)$ has at most one infinite cluster almost surely?
\end{question}

It is plausible that being dominated by a supercritical Bernoulli percolation (thus containing a.s. one infinite cluster) helps to prove uniqueness of the infinite cluster in the FK measure.

A negative answer would be somehow less unlikely in the $q<1$ regime.
In fact, interesting behaviors arise and are rigorously proven in the $q$ goes to 0 limit.
Besides translation-invariance, a key ingredient in the proof of Burton--Keane is the finite-energy property\footnote{A measure $\phi$ on $\{0,1\}^E$ is said to have the finite-energy property if a.s. $\phi\left(\omega(e)=1 \mid \omega_{\langle e\rangle}\right) \in (0,1)$ for all $e \in E$. It holds for all FK measures if and only if $p,p' \in (0,1)$, or equivalently $q\in(0, \infty)$ and $p'\in(0,1)$.}, which is lost in this limit.
We thus observe the appearance of measures, described below, with infinitely many infinite clusters.
Containing infinitely many infinite clusters seems to be a specific feature of models lacking the finite-energy property, but it is not clear how these measures at $q=0$ behave when $q$ is increased slightly above 0.

\bigskip

For all $\beta>0$, the $\beta$-\emph{arboreal gas} is a model of random spanning forest, defined as the limit when $q \to 0$ with $p=\beta q$ of the random cluster model. Equivalently, with our notations, it corresponds to $p=0$ and $p'=\beta/(1+\beta)>0$, that is the bottom line of Figures \ref{fig:phase_diag_Z2} and \ref{fig:phase_diagram_general}.
It is known that on $\bbZ^2$, the arboreal gas (for any $\beta>0$) never contains an infinite tree \cite{BCHS_2D}, while it does on $\bbZ^d, \ d\geq3$ for sufficiently large $\beta$ \cite{BCH_3D}. It strengthens the conjecture that in dimension 2 the critical line of FK percolation begins with the point $(p,p')=(0,1)$, whereas in dimension $d\geq3$, it begins with $(0,\beta_c/(1+\beta_c))$, $\beta_c$ being the critical parameter of the arboreal gas.
In \cite{HalberstamHutchcroft}, a Burton--Keane like statement is proved for the arboreal gas in dimension $d \leq 4$; namely, it is proved that any TI arboreal gas has a.s. at most one infinite tree.
This matches a well-known result of Pemantle \cite{Pemantle} for the \emph{uniform spanning tree} on $\bbZ^d, \ d\leq 4$, which is the $\beta \to \infty$ limit of the arboreal gas, or equivalently the FK model at $(p,p')=(0,1)$.
However, in dimension $d\geq5$, the uniform spanning tree contains a.s. infinitely many infinite trees, and this behavior is conjectured to occur in the arboreal gas for large enough $\beta$ as well.
Pemantle proved in \cite{Pemantle} that the free and wired uniform spanning tree measures always coincide on $\bbZ^d$, leading to the uniqueness of the infinite-volume measure. For the arboreal gas, similarly to the FK percolation with $q<1$, it is not known whether the free and wired measures are well-defined, and it is stressed by Halberstam and Hutchcroft \cite{HalberstamHutchcroft} that ``the structure of the set of Gibbs measures for the arboreal gas is very poorly understood''.

\begin{question}
	Is there a unique arboreal gas measure on $\bbZ^d$ for all $\beta>0$?
\end{question}
Uniqueness of the Gibbs measure is only known for $\beta$ small enough so that the model is dominated by a subcritical percolation, that is when $\beta \leq p_c(\bbZ^d)/(1-p_c(\bbZ^d)).$
Note that Theorem \ref{thm:subcritical_strong_uniqueness} also extends by some $\delta>0$ this range of uniqueness for the arboreal gas.

\printbibliography
\newpage

\end{document}